\documentclass[onefignum,onetabnum]{siamart190516_mod}

\usepackage{mathtools,amsfonts}
\usepackage[utf8]{inputenc}
\usepackage[T1]{fontenc}

\usepackage{enumitem}
\setlist[enumerate]{label=(\alph*)}

\usepackage[boxed]{algorithm2e}
\SetAlCapSkip{.5em}

\crefname{algocf}{Algorithm}{Algorithms}
\numberwithin{algocf}{section}

\renewcommand\phi\varphi
\newcommand\eps\varepsilon

\definecolor{todocolor}{rgb}{1.0,0.3,0.3}

\newcommand{\mscLink}[1]{\href{https://www.ams.org/mathscinet/msc/msc2020.html?t=#1}{#1}}

\newcommand{\PhiFB}{\Phi_{\textup{FB}}}
\newcommand{\cl}{\operatorname{cl}}
\newcommand{\downto}{\downarrow}
\newcommand{\LL}{\mathcal{L}}
\newcommand{\N}{\mathbb{N}}
\DeclareMathAlphabet{\mathpzc}{OT1}{pzc}{m}{it}
\newcommand{\oo}{\mathpzc{o}}
\newcommand{\OO}{\mathcal{O}}
\newcommand{\R}{\mathbb{R}}
\newcommand{\Span}{\operatorname{span}}

\newsiamremark{example}{Example}
\newsiamremark{remark}{Remark}

\newcommand{\skipline}{\hfill \par}

\newif\ifpreprint
\preprinttrue
\DeclarePairedDelimiter\myrbraces{.}{\}}
\providecommand\given{\nonscript\;\delimsize|\nonscript\;}
\DeclarePairedDelimiterX\set[1]{\{}{\}}{#1}
\DeclarePairedDelimiter\abs{\lvert}{\rvert}
\DeclarePairedDelimiter\norm{\lVert}{\rVert}
\DeclarePairedDelimiter\paren{(}{)}

\crefname{section}{Section}{Sections}
\crefname{subsection}{Subsection}{Subsections}

\allowdisplaybreaks

\headers{A Newton method for the M-stationarity system}{F. Harder, P. Mehlitz, G. Wachsmuth}

\title{Reformulation of the M-stationarity conditions  as a system of discontinuous equations and its solution by a semismooth Newton method\thanks{Submitted to the editors DATE.
\funding{This work is supported by the DFG Grant \emph{Bilevel Optimal Control: Theory, Algorithms, and Applications}
(Grant No.\ WA 3636/4-2) within the Priority Program SPP 1962 (Non-smooth
and Complementarity-based Distributed Parameter Systems: Simulation and Hierarchical Optimization).}}}

\author{%
	Felix Harder\thanks{%
		Brandenburgische Technische Universität Cottbus-Senftenberg, 
		Institute of Mathematics, 
		03046 Cott\-bus, Germany, 
		\url{https://www.b-tu.de/fg-optimale-steuerung}
	}
	\and
	Patrick Mehlitz\footnotemark[2]
	\and
	Gerd Wachsmuth\footnotemark[2]
}

\begin{document}
%%fakesection: Title, abstract und co

\maketitle

\begin{abstract}
	We show that the Mordukhovich-stationarity system
	associated with a mathematical program with complementarity constraints (MPCC)
	can be equivalently written as a system of discontinuous equations which can
	be tackled with a semismooth Newton method. 
	It will be demonstrated that the resulting algorithm
	can be interpreted as an active set strategy for MPCCs.
	Local fast convergence of the method is guaranteed under validity of an 
	MPCC-tailored version of LICQ and a suitable strong second-order condition. 
	In case of linear-quadratic MPCCs, the LICQ-type constraint qualification
	can be replaced by a weaker condition which depends on the underlying multipliers. 
	We discuss a suitable globalization strategy for our method.
	Some numerical results are presented in order to illustrate our theoretical
	findings.
\end{abstract}



\section{Introduction}

We aim for the numerical solution of so-called 
\emph{mathematical programs with complementarity constraints}
(MPCCs for short) which are nonlinear optimization problems of the form
\begin{equation}\label{eq:MPCC}\tag{MPCC}
	\begin{aligned}
		\min_{x} \quad& f(x)
		\\
		\text{s.t.}\quad&
		\begin{aligned}[t]
			g(x)	&\leq 	0,
			&
			h(x)	&=		0,
			\\
			G(x) 	&\geq	0,
			&
			H(x) 	&\geq 	0,
			&
			G(x)^\top H(x) &=	0.
		\end{aligned}
	\end{aligned}
\end{equation}
Throughout the article, we assume that the data functions $f\colon\R^n\to\R$, $g\colon\R^n\to\R^\ell$,
$h\colon\R^n\to\R^m$, and $G,H\colon\R^n\to\R^p$ are twice continuously differentiable.
Observing that most of the standard constraint qualifications fail to hold at the feasible points of
\eqref{eq:MPCC} while the feasible set of it is likely to be (almost) disconnected,
complementarity-constrained programs form a challenging class of optimization problems.
On the other hand, several real-world optimization scenarios from mechanics, finance, or
natural sciences naturally comprise equilibrium conditions which is why they can be modeled 
in the form \eqref{eq:MPCC}.
For an introduction to the topic of complementarity-constrained programming, the interested reader
is referred to the monographs \cite{LuoPangRalph1996,OutrataKocvaraZowe1998}.
In the past, huge effort has been put into the development of problem-tailored constraint qualifications
and stationarity notions which apply to \eqref{eq:MPCC}, see e.g.\ \cite{ScheelScholtes2000,Ye2005} for
an overview. Second-order necessary and sufficient optimality conditions for \eqref{eq:MPCC} are discussed
in \cite{Gfrerer2014,GuoLinYe2013,ScheelScholtes2000}.
There exist several different strategies in order to handle the inherent difficulties of \eqref{eq:MPCC}
in the context of its numerical solution. A common idea is to relax the complementarity constraints and
to solve the resulting standard nonlinear surrogate problems with a standard method, see e.g.\
\cite{HoheiselKanzowSchwartz2013} for an overview. Problem-tailored penalization approaches are discussed e.g.\ in 
\cite{HuRalph2004,HuangYangZhu2006,LeyfferLopezNocedal2006,RalphWright2004}.
Possible approaches for adapting the well-known SQP method of nonlinear programming to \eqref{eq:MPCC}
are investigated in \cite{BenkoGfrerer2016,FletcherLeyfferRalphScholtes2006,LuoPangRalph1998}.
Active set strategies for the numerical solution of \eqref{eq:MPCC} with affine complementarity constraints 
are under consideration in \cite{FukushimaTseng2002,JudiceSheraliRibeiroFaustino2007}.
In \cite{IzmailovPogosyanSolodov2012}, the authors combine a lifting approach as well as a
globalized semismooth Newton-type method in order to solve \eqref{eq:MPCC}.
Furthermore, we would like to mention the paper \cite{GuoLinYe2015} where the authors reformulate
different stationarity systems of \eqref{eq:MPCC} as (over-determined) nonlinear systems of equations subject
to a polyhedron, and the latter systems are solved via a Levenberg--Marquardt method. 

Using so-called NCP-functions, where NCP abbreviates \emph{nonlinear complementarity problem}, complementarity
restrictions can be transferred into systems of equations which are possibly nonsmooth.
Recall that a function $\pi\colon\R^2\to\R$ is called NCP-function whenever it satisfies
\[
	\forall (a,b)\in\R^2\colon\quad
	\pi(a,b)=0\,\Longleftrightarrow\,a,b\geq 0\,\land\,ab=0.
\]
Two popular examples of such NCP-functions are given by the minimum-function 
$\pi_\textup{min}\colon\R^2\to\R$ as
well as the Fischer--Burmeister-function $\pi_\textup{FB}\colon\R^2\to\R$ defined below:
\[
	\forall (a,b)\in\R^2\colon\quad
	\pi_\textup{min}(a,b):=\min(a,b),
	\quad
	\pi_\textup{FB}(a,b):=\sqrt{a^2+b^2}-a-b.
\]
A convincing overview of existing NCP-functions and their properties can be found in 
\cite{Galantai2012,KanzowYamashitaFukushima1997,SunQi1999}.
We note that most of the established NCP-functions like $\pi_\textup{min}$ or $\pi_\textup{FB}$ are nonsmooth.
Classically, NCP-function have been used to transfer Karush--Kuhn--Tucker (KKT) systems of standard nonlinear problems
with inequality constraints into systems of equations which then are tackled with the aid of
a Newton-type method which is capable of handling the potentially arising nonsmoothness, 
see \cite{DeLucaFacchineiKanzow1996,DeLucaFacchineiKanzow2000,FacchineiSoares1997} for an overview.
Furthermore, these papers report on the differentiability of the function $\pi_\textup{FB}^2$
which can be exploited in order to globalize the resulting Newton method.
In \cite{IzmailovSolodov2008}, the authors extended this idea to \eqref{eq:MPCC} by interpreting
it as a nonlinear problem. Under reasonable assumptions, local quadratic convergence to
so-called strongly stationary points has been obtained and suitable globalization strategies have
been presented.

In this paper, we aim to reformulate Mordukhovich's system of stationarity 
(the so-called system of M-stationarity) 
associated with \eqref{eq:MPCC} as a system of nonsmooth equations
which can be solved by a semismooth Newton method.
Our study is motivated by several different aspects. First, we would like to
mention that the set
\begin{equation}\label{eq:M_stationarity_set}
	M:=\set*{(a,b,\mu,\nu)\in\R^4\given 
		\begin{aligned}
			&0 \le a \perp b \ge 0, \,
			a \mu = 0,\;b \nu = 0, \\
			&(\mu\nu =0 \,\lor\, \mu<0,\nu<0)
		\end{aligned}
		\;
	},
\end{equation}
which is closely related to the M-stationarity system of \eqref{eq:MPCC},
see \cref{def:stationarity_systems}, is
closed. In contrast, the set
\begin{equation*}
	\tilde M:=\set*{(a,b,\mu,\nu)\in\R^4\given 
		\begin{aligned}
			&0 \le a \perp b \ge 0, \,
			a \, \mu = 0,\;b \, \nu = 0, \\
			&a=b=0\,\Rightarrow\,\mu,\nu \le 0
		\end{aligned}
		\;
		},
\end{equation*}
which similarly corresponds to
the system of strongly stationary points associated with
\eqref{eq:MPCC}, is not closed. In fact, $M$ is the closure of $\tilde M$.
Based on this topological observation, it is clear
that searching for M-stationary points is far more promising than searching
for strongly stationary points as long as both stationarity systems are
transferred into systems of nonsmooth equations which can be solved by 
suitable methods. 
In \cite{GuoLinYe2015}, the authors transferred the M-stationarity system of
\eqref{eq:MPCC} into a smooth (over-determined) system of equations subject to
a polyhedron, and they solved it with the aid of a modified Levenberg--Marquardt
method. It has been shown that the resulting algorithm converges quadratically to
an M-stationary point whenever an abstract error bound condition holds at the limit. 
Our aim in this
paper is to use a nonsmooth reformulation of the M-stationarity system which
can be tackled with a semismooth Newton method in order to ensure local
fast convergence of the resulting algorithm under suitable assumptions, 
namely MPCC-LICQ, an MPCC-tailored variant of the prominent Linear Independence Constraint
Qualification (LICQ), and MPCC-SSOC, an MPCC-tailored strong second-order condition,
have to hold at the limit point, see \cref{def:MPCC_LICQ,def:MPCC_SSC}
as well as \cref{thm:local_convergence}.
Using a continuously differentiable merit function,
we are in position to globalize our method,
see \cref{sec:glob}.
Observing that the strongly
stationary points of \eqref{eq:MPCC} can be found among its M-stationary 
points, the resulting method may converge to strongly stationary points, too. 
Let us mention that even in the absence of MPCC-LICQ, local
fast convergence of the method is possible if the linearly dependent gradients do not
appear in the Newton system, and, anyway, local slow convergence will be always
guaranteed via our globalization strategy. 
It will turn out that whenever the objective $f$ of \eqref{eq:MPCC} is quadratic
while the constraint functions $g$, $h$, $G$, and $H$ are affine, then we 
actually can replace MPCC-LICQ by a slightly weaker condition depending on the
multipliers at the limit point, see \cref{sec:linear_quadratic}.

The manuscript is organized as follows:
In \cref{sec:preliminaries}, we summarize the essential preliminaries.
Particularly, we recall some terminology from
complementarity-constrained programming and review the foundations of
semismooth Newton methods. 
\Cref{sec:M_St_as_system} is dedicated to
the reformulation of the M-stationarity system associated with \eqref{eq:MPCC}
as a system of nonsmooth equations. In order to guarantee that a Newton-type
method can be applied in order to solve the resulting system, we first
motivate the general structure of this system. Afterwards, we introduce
a so-called \emph{nonlinear M-stationarity function} whose roots are
precisely the elements of the set $M$ from \eqref{eq:M_stationarity_set}.
Although the nonlinear M-stationarity function of our interest
is nonsmooth and even discontinuous, we prove that it is Newton 
differentiable on the set $M$. Based on this function, we construct a
semismooth Newton method which solves the M-stationarity system of
\eqref{eq:MPCC} in \cref{sec:newton_method}. Furthermore, we provide a local
convergence analysis which shows that our method ensures local 
superlinear convergence
under validity of MPCC-LICQ and MPCC-SSOC at the limit point.
Local quadratic convergence can be achieved if, additionally,
the second-order derivatives of $f$, $g$, $h$, $G$, and $H$ are locally Lipschitz
continuous.
Moreover, we
illustrate that our Newton-type method can be interpreted as an active set
strategy for \eqref{eq:MPCC}. In \cref{sec:glob}, the globalization of the
algorithm is discussed. We exploit the standard idea to minimize 
a continuously differentiable merit function. 
In \cref{sec:linear_quadratic}, we show that it is possible to relax 
the requirement that MPCC-LICQ holds
in the setting of a linear-quadratic model problem \eqref{eq:MPCC}
while keeping all the desired convergence properties.
Numerical experiments are presented in \cref{sec:numerics}.
Some remarks close the paper in \cref{sec:conclusions}.

\section{Preliminaries}\label{sec:preliminaries}

\subsection{Notation}
We introduce the index sets
$I^\ell:=\{1,\ldots,\ell\}$, $I^m:=\{1,\ldots,m\}$, and $I^p:=\{1,\ldots,p\}$.
The component mappings of $g$, $h$, $G$, and $H$ are denoted by $g_i$ ($i\in I^\ell$), $h_i$ ($i\in I^m$),
$G_i$ ($i\in I^p$), and $H_i$ ($i\in I^p$), respectively.

We use $0$ in order to denote the scalar zero as well as the zero matrix of appropriate
dimensions. 
For $i\in\{1,\ldots,n\}$, we use $e_i\in\R^n$ to represent the $i$-th unit vector.
For a vector $v\in\R^n$ and a set $I\subset\{1,\ldots,n\}$, $v_I\in\R^{|I|}$ denotes the vector which results
from $v$ by deleting all entries corresponding to indices from $\{1,\ldots,n\}\setminus I$.
Similarly, for a matrix $V\in\R^{n\times m}$, $V_I\in\R^{|I|\times m}$ denotes the matrix which is obtained
by deleting all those rows from $V$ whose indices correspond to the elements of $\{1,\ldots,n\}\setminus I$.
If (row) vectors $v_i$, $i\in I$, are given, then $[v_i]_{I}$
denotes the matrix whose rows are precisely the vectors $v_i$, $i \in I$.
Finally, let us mention that for $x\in\R^n$ and $\varepsilon>0$, $B_\varepsilon(x)$ represents the
closed $\varepsilon$-ball around $x$.

\begin{lemma}\label{lem:span_of_polyhedral_cone}
	Let $A$ and $B$ be matrices of suitable dimensions satisfying
	\[
		A^\top\lambda+ B^\top\eta=0,\,\lambda\geq 0
		\quad
		\Longrightarrow
		\quad
		\lambda=0,
	\]
	i.e., the rows of $A$ are positive linearly independent from the rows of $B$.
	Then it holds
	\[
		\Span\{d\in\R^n\,|\,Ad\leq 0,\,Bd=0\}
		=
		\{d\in\R^n\,|\,Bd=0\}.
	\]
\end{lemma}
\begin{proof}
	For brevity, let us set
	\[
		C:=\{d\in\R^n\,|\,Ad\leq 0,\,Bd=0\}
		\qquad\text{and}\qquad
		Q:=\{d\in\R^n\,|\,Bd=0\}.
	\]
	Observing that these sets are polyhedral cones, we obtain
	\[
		C^\circ=\{A^\top\lambda+B^\top\eta\,|\,\lambda\geq 0\}
		\qquad\text{and}\qquad
		Q^\circ=B^\top\R^m
	\]
	where $X^\circ:=\{y\in\R^n\,|\,\forall x\in X\colon\,x^\top y\leq 0\}$ denotes
	the polar cone of a set $X\subset\R^n$.
	Clearly, it holds $Q^\circ\subset C^\circ\cap(-C^\circ)$.
	On the other hand, for $v\in C^\circ\cap(-C^\circ)$,
	we find vectors $\lambda^1,\lambda^2\geq 0$ and $\eta^1,\eta^2$
	satisfying $v=A^\top\lambda^1+B^\top\eta^1=-A^\top\lambda^2-B^\top\eta^2$.
	This leads to
	\[
		0=A^\top(\lambda^1+\lambda^2)+B^\top(\eta^1+\eta^2).
	\]
	The assumption of the lemma guarantees $\lambda^1+\lambda^2=0$.
	The sign condition on $\lambda^1$ and $\lambda^2$ yields $\lambda^1=\lambda^2=0$, 
	and $v\in Q^\circ$ follows.
	
	Due to $Q^\circ=C^\circ\cap(-C^\circ)$, polarization on both sides yields
	\[
		Q
		=
		Q^{\circ\circ}
		=
		(C^\circ\cap(-C^\circ))^\circ
		=
		\cl(C^{\circ\circ}-C^{\circ\circ})
		=
		\cl(C-C)
		=
		\cl(\Span C)
		=
		\Span C
	\]
	by the bipolar theorem,
	and this shows the claim.
\end{proof}
The next two lemmas are classical.
\begin{lemma}[{\cite[Lemma~16.1]{NocedalWright2006}}]
	\label{lem:saddle_point_system}
	Consider the saddle-point matrix
	\begin{equation*}
		C :=
		\begin{bmatrix}
			A & B^\top \\ B & 0
		\end{bmatrix}
		,
	\end{equation*}
	where $A$ and $B$ are matrices of compatible sizes.
	If the constraint block $B$ is surjective, i.e., the rows of $B$ are linearly independent,
	and if $A$ is positive definite on the kernel of $B$,
	i.e., $x^\top A \, x > 0$ for all $x \in \ker(B) \setminus \{0\}$,
	then $C$ is invertible.
\end{lemma}
\begin{lemma}[{\cite[p.~31]{Kato1995}}]
	\label{lem:neumann_series}
	Let the matrix $A$ be invertible.
	Then there exist constants $\varepsilon > 0$ and $C > 0$ such that
	\begin{equation*}
		\norm{(A + \delta A)^{-1}}
		\le
		C
	\end{equation*}
	holds for all matrices 
	$\delta A$ with $\norm{\delta A} \le \varepsilon$.
\end{lemma}

\subsection{MPCCs}

Here, we briefly summarize the well-known necessary essentials on stationarity conditions, 
constraint qualifications, and second-order conditions for complementarity-constrained 
optimization problems.
As mentioned earlier, most of the standard constraint qualifications do not hold at the
feasible points of \eqref{eq:MPCC} which is why stationarity notions, weaker than the
KKT conditions, have been introduced. Let us recall some of them.
For that purpose, we first introduce the MPCC-tailored Lagrangian 
$\LL\colon\R^n\times\R^\ell\times\R^m\times\R^p\times\R^p\to\R$ associated with
\eqref{eq:MPCC} via
\begin{align*}
	\LL(x,\lambda,\eta,\mu,\nu)
	:=
	f(x)+\lambda^\top g(x)+\eta^\top h(x)+\mu^\top G(x)+\nu^\top H(x).
\end{align*}
Furthermore, for a feasible point $\bar x\in\R^n$ of \eqref{eq:MPCC}, we will make use
of the index sets
\begin{align*}
	I^g(\bar x)		&:=\{i\in I^\ell\,|\,g_i(\bar x)=0\},\\
	I^{+0}(\bar x)	&:=\{i\in I^p\,|\,G_i(\bar x)>0\,\land\,H_i(\bar x)=0\},\\
	I^{0+}(\bar x)	&:=\{i\in I^p\,|\,G_i(\bar x)=0\,\land\,H_i(\bar x)>0\},\\
	I^{00}(\bar x)	&:=\{i\in I^p\,|\,G_i(\bar x)=0\,\land\,H_i(\bar x)=0\}.
\end{align*}
Clearly, $\{I^{+0}(\bar x),I^{0+}(\bar x),I^{00}(\bar x)\}$ is a disjoint partition of $I^p$.
\begin{definition}\label{def:stationarity_systems}
	Let $\bar x\in\R^n$ be a feasible point of \eqref{eq:MPCC}.
	Then $\bar x$ is said to be
	\begin{enumerate}
		\item Mordukhovich-stationary (M-stationary) if there exist multipliers
			$\lambda\in\R^\ell$, $\eta\in\R^m$, and $\mu,\nu\in\R^p$ which solve the system
			\begin{subequations}\label{eq:M_St}
				\begin{align}
					\label{eq:M_St_x}
						&\nabla_x\LL(\bar x,\lambda,\eta,\mu,\nu)=0,\\
					\label{eq:M_St_lambda}
						&\lambda_{I^g(\bar x)}\geq 0,\quad\lambda_{I^\ell\setminus I^g(\bar x)}=0,\\
					\label{eq:M_St_mu}
						&\mu_{I^{+0}(\bar x)}=0,\\
					\label{eq:M_St_nu}
						&\nu_{I^{0+}(\bar x)}=0,\\
					\label{eq:M_St_I00}
						&\forall i\in I^{00}(\bar x)\colon\,
						 \mu_i\nu_i=0\,\lor\,(\mu_i<0\,\land\,\nu_i<0),
				\end{align}
			\end{subequations}
		\item strongly stationary (S-stationary) if there exist multipliers
			$\lambda\in\R^\ell$, $\eta\in\R^m$, and $\mu,\nu\in\R^p$ which satisfy
			\eqref{eq:M_St_x}-\eqref{eq:M_St_nu} and
			\begin{equation}\label{eq:S_St}
				\mu_{I^{00}(\bar x)}\leq 0,\quad\nu_{I^{00}(\bar x)}\leq 0.
			\end{equation}
	\end{enumerate}
\end{definition}

Let us briefly note that there exist several more stationarity notions which apply to
\eqref{eq:MPCC}, see e.g.\ \cite{Ye2005} for an overview.
For later use, let $\Lambda^\textup{M}(\bar x)$ and $\Lambda^\textup{S}(\bar x)$ be the sets 
of all multipliers which solve the system of M- and S-stationarity w.r.t.\ a 
feasible point $\bar x\in\R^n$ of \eqref{eq:MPCC}, respectively.

In this paper, we will make use of a popular MPCC-tailored version of the Linear Independence
Constraint Qualification.
\begin{definition}\label{def:MPCC_LICQ}
	Let $\bar x\in\R^n$ be a feasible point of \eqref{eq:MPCC}. 
	Then the MPCC-tailored Linear Independence Constraint Qualification (MPCC-LICQ)
	is said to hold at $\bar x$ whenever the matrix
	\begin{equation*}
		\begin{bmatrix}
			g'(\bar x)_{I^g(\bar x)}
			\\
			h'(\bar x)
			\\
			G'(\bar x)_{I^{0+}(\bar x)\cup I^{00}(\bar x)}
			\\
			H'(\bar x)_{I^{+0}(\bar x)\cup I^{00}(\bar x)}
		\end{bmatrix}
	\end{equation*}
	possesses full row rank.
\end{definition}

It is a classical result that a local minimizer of \eqref{eq:MPCC} where MPCC-LICQ holds is
S-stationary. Furthermore, the associated multipliers $(\lambda,\eta,\mu,\nu)$, which
solve the system \eqref{eq:M_St_x}-\eqref{eq:M_St_nu}, \eqref{eq:S_St} are uniquely
determined in this case. 
It has been reported in \cite{FlegelKanzow2006:1}
that even under validity
of mild MPCC-tailored constraint qualifications, local minimizers of
\eqref{eq:MPCC} are M-stationary. Therefore, it is a reasonable strategy to identify the
M-stationary points of a given complementarity-constrained optimization problem in order
to tackle the problem of interest.

We review existing second-order optimality conditions addressing \eqref{eq:MPCC}
which are based on S-stationary points.
We adapt the considerations from \cite{ScheelScholtes2000}.
For some point $\bar x\in\R^n$, we first introduce the so-called MPCC-critical cone
\[
	\mathcal C(\bar x)
	:=
	\left\{
		\delta x\in\R^n
		\,\middle|\,
			\begin{aligned}
				\nabla f(\bar x)^\top\delta x&\leq 0&\,&\\
				\nabla g_i(\bar x)^\top\delta x&\leq 0&\,&i\in I^g(\bar x)\\
				h'(\bar x)\delta x&=0&&\\
				\nabla G_i(\bar x)^\top\delta x&=0&&i\in I^{0+}(\bar x)\\
				\nabla H_i(\bar x)^\top\delta x&=0&&i\in I^{+0}(\bar x)\\
				\nabla G_i(\bar x)^\top\delta x&\geq 0&&i\in I^{00}(\bar x)\\
				\nabla H_i(\bar x)^\top\delta x&\geq 0&&i\in I^{00}(\bar x)\\
				(\nabla G_i(\bar x)^\top\delta x)(\nabla H_i(\bar x)^\top\delta x)&=0&&i\in I^{00}(\bar x)
			\end{aligned}
	\right\}.
\]
We note that this cone is likely to be not convex if the index set $I^{00}(\bar x)$ of biactive 
complementarity constraints
is nonempty. In case where $\bar x$ is an S-stationary point of \eqref{eq:MPCC} 
and $(\lambda,\eta,\mu,\nu)\in\Lambda^\textup{S}(\bar x)$ is arbitrarily chosen, we obtain the representation
\[
	\mathcal C(\bar x)
	=
	\left\{
		\delta x\in\R^n
		\,\middle|\,
			\begin{aligned}
				\nabla g_i(\bar x)^\top\delta x&= 0&\,&i\in I^g(\bar x),\lambda_i>0\\
				\nabla g_i(\bar x)^\top\delta x&\leq 0&\,&i\in I^g(\bar x),\lambda_i=0\\
				h'(\bar x)\delta x&=0&&\\
				\nabla G_i(\bar x)^\top\delta x&=0&&i\in I^{0+}(\bar x)\cup I^{00}_{\pm\R}(\bar x,\mu,\nu)\\
				\nabla H_i(\bar x)^\top\delta x&=0&&i\in I^{+0}(\bar x)\cup I^{00}_{\R\pm}(\bar x,\mu,\nu)\\
				\nabla G_i(\bar x)^\top\delta x&\geq 0&&i\in I^{00}_{00}(\bar x,\mu,\nu)\\
				\nabla H_i(\bar x)^\top\delta x&\geq 0&&i\in I^{00}_{00}(\bar x,\mu,\nu)\\
				(\nabla G_i(\bar x)^\top\delta x)(\nabla H_i(\bar x)^\top\delta x)&=0&&i\in I^{00}_{00}(\bar x,\mu,\nu) 
			\end{aligned}
	\right\}
\]
by elementary calculations, see \cite[Lemma~4.1]{Mehlitz2019b} as well, where we used
\begin{subequations}
	\label{eq:more_index_sets}
	\begin{align}
		I^{00}_{\pm\R}(\bar x,\mu,\nu)&:=\{i\in I^{00}(\bar x)\,|\,\mu_j\neq 0\},\\
		I^{00}_{\R\pm}(\bar x,\mu,\nu)&:=\{i\in I^{00}(\bar x)\,|\,\nu_j\neq 0\},\\
		I^{00}_{00}(\bar x,\mu,\nu)&:=\{i\in I^{00}(\bar x)\,|\,\mu_j=\nu_j=0\}.
	\end{align}
\end{subequations}
If $\bar x\in\R^n$ is a local minimizer of \eqref{eq:MPCC} where MPCC-LICQ holds, then the unique
multiplier $(\lambda,\eta,\mu,\nu)\in\Lambda^\textup{S}(\bar x)$ satisfies
\[
	\forall \delta x\in\mathcal C(\bar x)\colon
	\quad
	\delta x^\top\nabla^2_{xx}\LL(\bar x,\lambda,\eta,\mu,\nu)\delta x\geq 0.
\]
Let us note that necessary second-order conditions for \eqref{eq:MPCC} which are based on 
M-stationary points can be found in \cite{GuoLinYe2013}.
On the other hand, if $\bar x$ is an arbitrary S-stationary point of \eqref{eq:MPCC} where the
so-called MPCC-tailored Second-Order Sufficient Condition (MPCC-SOSC) given by
\[
	\forall\delta x\in\mathcal C(\bar x)\setminus\{0\}\,
	\exists (\lambda,\eta,\mu,\nu)\in\Lambda^\textup{S}(\bar x)\colon
	\qquad
	\delta x^\top\nabla^2_{xx}\LL(\bar x,\lambda,\eta,\mu,\nu)\delta x>0
\]
holds, then $\bar x$ is a strict local minimizer of \eqref{eq:MPCC}.
More precisely, the second-order growth condition holds for \eqref{eq:MPCC} at $\bar x$.

Finally, we are going to state the second-order condition which we are going to exploit for
our convergence analysis. 
As we will see later, it generalizes a strong second-order condition (SSOC) exploited in order to
ensure local fast convergence of semismooth Newton-type methods for the numerical solution
of KKT systems associated with standard nonlinear programs, see
\cite[Section~5.2]{FacchineiFischerKanzowPeng1999} which is based on the theory from
\cite[Section~4]{Robinson1980}.
Observe that the subsequent definition is based on M-stationary points.
\begin{definition}\label{def:MPCC_SSC}
	Let $\bar x\in\R^n$ be an M-stationary point of \eqref{eq:MPCC}.
	Furthermore, let $(\lambda,\eta,\mu,\nu)\in\Lambda^\textup{M}(\bar x)$
	be fixed.
	Then the MPCC-tailored Strong Second-Order Condition (MPCC-SSOC)
	is said to hold at $\bar x$ w.r.t.\ $(\lambda,\eta,\mu,\nu)$ whenever
	\[
		\forall \delta x\in S(\bar x,\lambda,\mu,\nu)\setminus\{0\}\colon\quad
		\delta x^\top\nabla^2_{xx}\LL(\bar x,\lambda,\eta,\mu,\nu)\delta x
		>0
	\]
	holds true. 
	Here, the set $S(\bar x,\lambda,\mu,\nu) \subset \R^n$ is given by
	\[
		S(\bar x,\lambda,\mu,\nu)
		:=
		\left\{
			\delta x
			\,\middle|\,
			\begin{aligned}
				\nabla g_i(\bar x)^\top\delta x&=0&&i\in I^g(\bar x),\,\lambda_i>0\\
				h'(\bar x)\delta x&=0&&\\
				\nabla G_i(\bar x)^\top\delta x&=0&&i\in I^{0+}(\bar x)\cup I^{00}_{\pm\R}(\bar x,\mu,\nu)\\
				\nabla H_i(\bar x)^\top\delta x&=0&&i\in I^{+0}(\bar x)\cup I^{00}_{\R\pm}(\bar x,\mu,\nu)\\
				(\nabla G_i(\bar x)^\top\delta x)(\nabla H_i(\bar x)^\top\delta x)&=0&&i\in I^{00}_{00}(\bar x,\mu,\nu)
			\end{aligned}
		\right\}.
	\]
\end{definition}

Fix a feasible point $\bar x\in\R^n$ of \eqref{eq:MPCC} which is M-stationary and let
an associated multiplier
$(\lambda,\eta,\mu,\nu)\in\Lambda^\textup{M}(\bar x)$ be given.
Let us clarify that validity of MPCC-SSOC at $\bar x$ does \emph{not} provide a 
sufficient optimality condition for
\eqref{eq:MPCC} in general since M-stationarity
does not rule out the presence of descent directions at the underlying point of interest.
Furthermore, the set $S(\bar x,\lambda,\mu,\nu)$ seems to be too large for the purpose of
being used in order to derive second-order necessary optimality conditions of \eqref{eq:MPCC}
based on M-stationary points, see \cite[Section~3]{GuoLinYe2013}.
For any set $\beta\subset I^{00}_{00}(\bar x,\mu,\nu)$,
we define the complement
$\bar\beta:=I^{00}_{00}(\bar x,\mu,\nu) \setminus \beta$ as well as
\[
	S_\beta(\bar x,\lambda,\mu,\nu)
	:=
	\left\{
			\delta x\in\R^n
			\,\middle|\,
			\begin{aligned}
				\nabla g_i(\bar x)^\top\delta x&=0&\;&i\in I^g(\bar x),\,\lambda_i>0\\
				h'(\bar x)\delta x&=0&&\\
				\nabla G_i(\bar x)^\top\delta x&=0&&i\in I^{0+}(\bar x)\cup I^{00}_{\pm\R}(\bar x,\mu,\nu)\cup\beta\\
				\nabla H_i(\bar x)^\top\delta x&=0&&i\in I^{+0}(\bar x)\cup I^{00}_{\R\pm}(\bar x,\mu,\nu)\cup\bar\beta
			\end{aligned}
		\right\}.
\]
Then we have 
\[
	S(\bar x,\lambda,\mu,\nu)=\bigcup\limits_{\beta\subset I^{00}_{00}(\bar x,\mu,\nu)}S_\beta(\bar x,\lambda,\mu,\nu).
\]
Thus, MPCC-SSOC holds at $\bar x$ w.r.t.\ $(\lambda,\eta,\mu,\nu)\in\Lambda^\textup{M}(\bar x)$
if and only if the classical SSOC 
from \cite{FacchineiFischerKanzowPeng1999,Robinson1980} is valid at $\bar x$ along all the
NLP branches of \eqref{eq:MPCC} given by
\begin{equation*}
	\begin{aligned}
		\min_{x} \quad& f(x)
		\\
		\text{s.t.}\quad&
		\begin{aligned}[t]
			g(x)	&\leq 	0,
			&
			h(x)	&=		0,
			&
			&&
			\\
			G_i(x) 	&\geq	0,
			&
			H_i(x) 	&= 	0,
			&
			&
			i\in I^{+0}(\bar x)\cup I^{00}_{\R\pm}(\bar x,\mu,\nu)\cup\bar\beta
			&
			\\
			G_i(x)	&=	0,
			&
			H_i(x)	&\geq	0,
			&
			&
			i\in I^{0+}(\bar x)\cup I^{00}_{\pm\R}(\bar x,\mu,\nu)\cup\beta
			&
		\end{aligned}
	\end{aligned}
\end{equation*}
for $\beta\subset I^{00}_{00}(\bar x,\mu,\nu)$. Therefore, MPCC-SSOC provides
a reasonable generalization of the SSOC to \eqref{eq:MPCC}.
Furthermore, due to \cref{lem:span_of_polyhedral_cone}, under validity of MPCC-LICQ at $\bar x$, we have 
$S_\beta(\bar x,\lambda,\mu,\nu)=\Span \mathcal C_\beta(\bar x,\lambda,\mu,\nu)$ where we used
\[
	\mathcal C_\beta(\bar x,\lambda,\mu,\nu)
	:=
	\left\{
		\delta x\in\R^n
		\,\middle|\,
			\begin{aligned}
				\nabla g_i(\bar x)^\top\delta x&= 0&\,&i\in I^g(\bar x),\lambda_i>0\\
				\nabla g_i(\bar x)^\top\delta x&\leq 0&\,&i\in I^g(\bar x),\lambda_i=0\\
				h'(\bar x)\delta x&=0&&\\
				\nabla G_i(\bar x)^\top\delta x&=0&&i\in I^{0+}(\bar x)\cup I^{00}_{\pm\R}(\bar x,\mu,\nu)\cup\beta\\
				\nabla H_i(\bar x)^\top\delta x&=0&&i\in I^{+0}(\bar x)\cup I^{00}_{\R\pm}(\bar x,\mu,\nu)\cup\bar\beta\\
				\nabla G_i(\bar x)^\top\delta x&\geq 0&&i\in \bar\beta\\
				\nabla H_i(\bar x)^\top\delta x&\geq 0&&i\in \beta
			\end{aligned}
	\right\},
\]
i.e., $S(\bar x,\lambda,\mu,\nu)$ is the finite union of the spans of polyhedral cones.
Observing that
\[
	\mathcal C(\bar x)
	=
	\bigcup\limits_{\beta\subset I^{00}_{00}(\bar x,\mu,\nu)}\mathcal C_\beta(\bar x,\lambda,\mu,\nu)
\]
holds true provided $\bar x$ is an S-stationary point while $(\lambda,\eta,\mu,\nu)\in\Lambda^\textup{S}(\bar x)$
holds, the set $S(\bar x,\lambda,\mu,\nu)$ is closely related to the critical cone of \eqref{eq:MPCC}.
Again, this underlines that the name MPCC-SSOC in \cref{def:MPCC_SSC} is quite reasonable.
Further observe that the inclusion $\mathcal C(\bar x)\subset S(\bar x,\lambda,\mu,\nu)$
holds for each S-stationary point $\bar x$ and each multiplier $(\lambda,\eta,\mu,\nu)\in\Lambda^\textup{S}(\bar x)$,
i.e., MPCC-SSOC is slightly stronger than MPCC-SOSC in this situation.

\subsection{Semismooth Newton methods}\label{sec:nonsmooth_Newton}
In this section,
we collect some theory
concerning the application of Newton methods
for functions $F \colon \R^n \to \R^n$ which are not continuously differentiable.
In the finite-dimensional case,
one typically utilizes semismooth functions.
Since semismooth functions are by definition locally Lipschitz continuous,
this theory is not applicable to discontinuous functions.
Hence,
we exploit the concept of Newton differentiability,
which is used in infinite-dimensional applications of Newton's method,
see \cite{ChenNashedQi2000,HintermuellerItoKunisch2002,Ulbrich2002,ItoKunisch2008}.
\begin{definition}
	\label{def:newton_differentiability}
	Let $F \colon \R^n \to \R^m$
	and $DF \colon \R^n \to \R^{m \times n}$ be given.
	The function $F$ is said to be Newton differentiable
	(with derivative $DF$)
	on a set
	$K \subset \R^n$ if
	\begin{equation*}
		F(x + h) - F(x) - DF(x + h)\,h = \oo(\norm{h})
		\qquad\text{for } h \to 0
	\end{equation*}
	holds for all $x \in K$.
	For $\alpha \in (0,1]$, the function $F$ is Newton differentiable
	of order $\alpha$, if
	\begin{equation*}
		F(x + h) - F(x) - DF(x + h)\,h = \OO(\norm{h}^{1+\alpha})
		\qquad\text{for } h \to 0
	\end{equation*}
	holds for all $x \in K$.
	Finally, $F$ is said to be Newton differentiable of order $\infty$,
	if for all $x \in K$ there is $\varepsilon_x > 0$
	such that
	\begin{equation*}
		\forall h \in B_{\varepsilon_x}(0)\colon\quad
		F(x + h) - F(x) - DF(x + h)\,h = 0.
	\end{equation*}
\end{definition}
Clearly,
if $F$ is continuously differentiable,
then $DF = F'$ is a Newton derivative.
If $F'$ is locally Lipschitz continuous, then $F$ is Newton differentiable of order 1.

In the following example, we discuss the Newton differentiability of
the minimum and maximum entry of a vector.
Note that this particular choice for the Newton derivative
will be essential for our argumentation in the later parts of this paper.
\begin{example}\label{ex:Newton_differentiability_of_min_max}
For the nonsmooth functions
$\min,\max\colon\R^n\to\R$, we establish the following convention
for choosing Newton derivatives at arbitrary points $a\in\R^n$:
\begin{equation}\label{eq:newton_derivative_min_max}
	\begin{aligned}
		D\min(a_1,\ldots,a_n) &:=
		e_i^\top,\quad
		i = \min\set[\big]{j\in\set{1,\ldots,n}\given a_j=\min(a_1,\ldots,a_n)},
		\\
		D\max(a_1,\ldots,a_n) &:=
		e_i^\top,\quad
		i = \min\set[\big]{j\in\set{1,\ldots,n}\given a_j=\max(a_1,\ldots,a_n)},
	\end{aligned}
\end{equation}
i.e., we give priority to variables that appear first in a $\min$
or $\max$ expression.
\ifpreprint
Let us verify that this choice ensures that $\min$ is 
indeed Newton differentiable of order $\infty$. Similar arguments apply in order
to show the same properties of $\max$.

For arbitrary $a\in\R^n$, we introduce $I(a):=\{j\in\{1,\ldots,n\}\,|\,a_j=\min(a_1,\ldots,a_n)\}$.
Let $i_0\in I(a)$ be fixed.
By definition of the minimum there is some $\varepsilon>0$ such that we have
\begin{equation*}
	\forall h\in\R^n\colon\quad
	\norm{h}<\varepsilon
	\,\Longrightarrow\,
	\min(a_1+h_1,\ldots,a_n+h_n)=\min\{a_j+h_j\,|\,j\in I(a)\}.
\end{equation*}
Thus, for each $h\in\R^n$ satisfying $\norm{h}<\varepsilon$, it holds 
$D\min(a_1+h_1,\ldots,a_n+h_n)=e_{i_h}^\top$ where $i_h\in\{1,\ldots,n\}$ satisfies
\begin{align*}
	i_h
	&
	=
	\min\{i\in\{1,\ldots,n\}\,|\,a_i+h_i=\min(a_1+h_1,\ldots,a_n+h_n)\}\\
	&
	=
	\min\{i\in I(a)\,|\,a_i+h_i=\min\{a_j+h_j\,|\,j\in I(a)\}\}\\
	&
	=
	\min\{i\in I(a)\,|\,a_{i_0}+h_i=\min\{a_{i_0}+h_j\,|\,j\in I(a)\}\}\\
	&
	=
	\min\{i\in I(a)\,|\,h_i=\min\{h_j\,|\,j\in I(a)\}\}
\end{align*}
due to the above priority rule.
On the other hand, for the same $h$, it holds
\begin{align*}
	\min(a_1+h_1,\ldots,a_n+h_n)-\min(a_1,\ldots,a_n)
	&
	=
	\min\{a_j+h_j\,|\,j\in I(a)\}-a_{i_0}\\
	&
	=
	\min\{h_j\,|\,j\in I(a)\}\\
	&
	=
	D\min(a_1+h_1,\ldots,a_n+h_n)\,h
\end{align*}
due to the above arguments, i.e., $\min$ is indeed Newton differentiable of order $\infty$.
\else
	With this choice, the functions $\min$ and $\max$
	are Newton differentiable of order $\infty$.
\fi
\end{example}

In order to find a solution $\bar x$ of $F(\bar x) = 0$ where $F\colon\R^n\to\R^n$
is a Newton differentiable map,
we use the iteration
\begin{equation}\label{eq:Newton_method}
	x_{k+1} := x_k - DF(x_k)^{-1} F(x_k)
	,\qquad
	k = 0,1,\ldots
\end{equation}
for an initial guess $x_0\in\R^n$.
As usual, we call this iteration scheme \emph{semismooth Newton method} but
emphasize that it applies to mappings $F$ which are not semismooth
in the classical sense.

Nowadays, the proof of the next theorem is classical,
see, e.g., \cite[proof of Theorem~3.4]{ChenNashedQi2000}.
\begin{theorem}
	\label{thm:newton_method}
	Assume that $F \colon \R^n \to \R^n$ is Newton differentiable on $K \subset \R^n$
	with Newton derivative $DF$.
	Further assume that $\bar x \in K$ satisfies $F(\bar x) = 0$
	and that the matrices $\set{DF(x) \given x \in B_\varepsilon(\bar x)}$
	are uniformly invertible for some $\varepsilon > 0$.
	Then there exists $\delta > 0$ such that the
	Newton-type iteration \eqref{eq:Newton_method} is well defined for any
	initial iterate $x_0 \in B_\delta(\bar x)$ while the 
	associated sequence of iterates converges superlinearly.
	If $F$ is additionally Newton differentiable of order $1$,
	the convergence is quadratic,
	and we have convergence in one step if $F$ is Newton differentiable of order $\infty$.
\end{theorem}
\begin{proof}
	We choose $\delta > 0$ and $M > 0$
	such that
	\begin{equation*}
		\norm{DF(x)^{-1}} \le M
		\quad\text{and}\quad
		\norm{ F(x) - F(\bar x) - DF(x)\,(x - \bar x)} \le \frac1{2\,M} \, \norm{ x - \bar x}
	\end{equation*}
	hold for all $x \in B_\delta(\bar x)$.
	For $x_k \in B_\delta(\bar x)$, we obtain
	\begin{equation}
		\label{eq:newton_proof}
		\begin{aligned}
			x_{k+1} - \bar x
			&=
			x_k - DF(x_k)^{-1} F(x_k)
			- \bar x
			\\
			&=
			-
			DF(x_k)^{-1} \,
			\paren[\big]{ F(x_k) - F(\bar x) - DF(x_k)\,(x_k - \bar x) }
			.
		\end{aligned}
	\end{equation}
	Thus, $\norm{x_{k+1} - \bar x} \le \norm{x_k - \bar x}/2$ follows.
	This shows that the iteration is well defined
	for any $x_0 \in B_\delta(\bar x)$
	and $x_k \to \bar x$.
	Now, \eqref{eq:newton_proof} together with
	the required order of
	Newton differentiability of $F$
	implies
	\begin{align*}
		\norm{x_{k+1} - \bar x}
		&\le M \, \norm{ F(x_k) - F(\bar x) - DF(x_k)\,(x_k - \bar x) }
		\\&
		=
		\begin{cases}
			\OO(\norm{x_k - \bar x}^2) & \text{differentiability of order $1$}, \\
			0 & \text{differentiability of order $\infty$}, \\
			\oo(\norm{x_k - \bar x}) & \text{else},
		\end{cases}
	\end{align*}
	where $\delta$ may need to be reduced in the case where
	the order of Newton differentiability equals $\infty$.
	This shows the claim.
\end{proof}
If the assumptions of \cref{thm:newton_method} are satisfied,
one obtains the equivalence of the known residuum $\norm{F(x)}$
and the unknown distance $\norm{x - \bar x}$.
\begin{corollary}
	\label{lem:norm_equivalence}
	In addition to the assumptions of \cref{thm:newton_method},
	suppose that the matrices $\set{DF(x) \given x \in B_\varepsilon(\bar x)}$ are bounded
	for some $\varepsilon > 0$.
	Then there exist constants $c, C, \delta > 0$ such that
	\begin{equation*}
		\forall x \in B_\delta(\bar x)\colon\quad
		c \, \norm{F(x)}
		\le
		\norm{x - \bar x}
		\le
		C \, \norm{F(x)}.
	\end{equation*}
\end{corollary}
\ifpreprint
	\begin{proof}
		Due to the Newton differentiability of $F$, we have
		\begin{equation*}
			F(x)
			=
			F(x) - F(\bar x)
			=
			DF(x) \, (x - \bar x)
			+
			\oo(\norm{x - \bar x})
			\qquad\text{for } x \to \bar x.
		\end{equation*}
		Exploiting the properties of $DF$ and the postulated assumptions,
		there are $\delta>0$ and $C_1,C_2>0$ such that
		we have
		\begin{equation*}
			C_1\norm{x-\bar x}
			\le 
			\norm{DF(x)^{-1}}^{-1} \, \norm{x - \bar x}
			\le
			\norm{DF(x) \, (x - \bar x)}
			\le
			\norm{DF(x)} \, \norm{x - \bar x}
			\le 
			C_2\norm{x-\bar x}
		\end{equation*}
		as well as $\norm{\oo(\norm{x-\bar x})}\leq \tfrac{C_1}{2}\norm{x-\bar x}$
		for all $x\in B_\delta(\bar x)$.
		The claim follows choosing $c:=2/C_1$ and $C:=2/(C_1+2C_2)$.
	\end{proof}
\else
	The proof follows from
	\begin{equation*}
		F(x)
		=
		F(x) - F(\bar x)
		=
		DF(x) \, (x - \bar x)
		+
		\oo(\norm{x - \bar x})
	\end{equation*}
	and the properties of $DF(x)$.
\fi
For the reader's convenience, we provide the following chain rule.
\begin{lemma}
	\label{lem:chain_rule}
	Suppose that $f \colon \R^n \to \R^m$ is Newton differentiable
	on $K \subset \R^n$
	with derivative $Df$
	and that $g \colon \R^m \to \R^p$ is Newton differentiable
	on $f(K)$ with derivative $Dg$.
	Further, we assume that $Df$ is bounded on a neighborhood of $K$
	and that $Dg$ is bounded on a neighborhood of $f(K)$.
	Then $g \circ f$ is Newton differentiable on $K$
	with derivative given by
	$x \mapsto Dg(f(x)) \, Df(x)$.
	If both $f$ and $g$ are Newton differentiable of order 
	$\alpha \in (0,1] \cup \{\infty\}$,
	then $g \circ f$ is Newton differentiable of order $\alpha$.
\end{lemma}
\begin{proof}
	We follow the proof of \cite[Theorem~9.3]{Clason2017}.
	We define the remainder term $r_f$ of $f$ via
	\begin{equation*}
		r_f(x;h) := f(x + h) - f(x) - Df(x + h) \, h.
	\end{equation*}
	Similarly, we define $r_g$ and $r_{g\circ f}$.
	The Newton differentiability of $f$ together with the boundedness of $Df$
	implies
	\begin{equation*}
		f(x + h) - f(x)
		=
		Df(x + h) \, h + \oo(\norm{h})
		=
		\OO(\norm{h})
		\qquad\text{as } h \to 0.
	\end{equation*}
	In particular, $f(x+h) - f(x) \to 0$ as $h \to 0$.
	Next, we have
	\begin{align*}
		r_{g\circ f}(x; h)
		&=
		g(f(x+h)) - g(f(x)) - Dg(f(x + h)) \, Df(x + h) \, h
		\\
		&=
		r_g( f(x); f(x+h) - f(x)) \\
		&\qquad+ Dg(f(x+h)) \, \paren{f(x+h) - f(x) - Df(x+h)\,h}
		\\
		&=
		r_g( f(x); f(x+h) - f(x)) + Dg(f(x+h)) \, r_f(x; h)
		.
	\end{align*}
	Now, the boundedness of $Dg(f(x+h))$
	implies
	\begin{equation*}
		r_{g\circ f}(x; h)
		=
		\oo(\norm{f(x+h) - f(x)}) + \norm{Dg(f(x+h))} \, \oo(\norm{h})
		=
		\oo(\norm{h}).
	\end{equation*}
	Similarly, if $f$ and $g$ are Newton differentiable of order $\alpha \in (0,1]$,
	we get
	\begin{equation*}
		r_{g\circ f}(x; h)
		=
		\OO(\norm{f(x+h) - f(x)}^\alpha) + \norm{Dg(f(x+h))} \, \OO(\norm{h}^\alpha)
		=
		\OO(\norm{h}^\alpha).
	\end{equation*}
	Finally, if both functions are Newton differentiable of order $\infty$,
	we get $r_{g \circ f}(x; h) = 0$ if $h$ is small enough.
\end{proof}
\begin{example}\label{ex:Newton_differentiability_of_absolute_value}
	Exploiting \cref{ex:Newton_differentiability_of_min_max} and
	\cref{lem:chain_rule}, the absolute value function $\abs{\cdot}\colon\R\to\R$
	is Newton differentiable of order $\infty$,
	since $\abs{x}=\max(x,-x)$ for each $x\in\R$.
	Following the convention from \cref{ex:Newton_differentiability_of_min_max},
	the associated
	Newton derivative is given by
	\begin{equation}
		\label{eq:newton_derivative_absolute_value}
			\forall x\in\R\colon\quad
			D|\cdot|(x)
			=
			\begin{cases}
				1	&\text{if}\;x\geq 0,\\
				-1	&\text{if}\;x<0.				
			\end{cases}
	\end{equation}	
\end{example}

\section{M-Stationarity as a nonlinear system of equations}\label{sec:M_St_as_system}

\subsection{Preliminary considerations}\label{sec:preliminaries_NMSt_function}

As stated before, we want to reformulate the M-stationarity system 
\eqref{eq:M_St}
as an equation.
To this end,
we need to encode the complementarity conditions
\eqref{eq:M_St_lambda}
and the conditions \eqref{eq:M_St_mu}--\eqref{eq:M_St_I00}, which depend on index sets,
as the zero level set of a suitable function.
For clarity of the presentation,
we temporarily consider the simplified MPCC problem 
\begin{equation*}
	\label{eq:simple_MPCC}
	\tag{MPCC1}
	\begin{aligned}
		\min_{x} \quad& f(x)
		\\
		\text{s.t.}\quad&
		\begin{aligned}[t]
			0\leq H(x) \perp G(x) &\geq 0,
		\end{aligned}
	\end{aligned}
\end{equation*}
with $G,H\colon\R^n\to\R$, i.e., there is only one complementarity constraint.
This simplification will also ease notation in this section.
The results of this section will be transferred to the problem
\eqref{eq:MPCC} with $p$ many complementarity conditions in
\cref{sec:newton_method}.
The M-stationarity system for \eqref{eq:simple_MPCC} is given by
\begin{subequations}
	\label{eq:simp_MPCC}
	\begin{align}
		\nabla_x\LL(x,\mu,\nu) &=0,\\
		\label{eq:simp_MPCC_feasible}
		0\leq H(x) \perp G(x) &\geq 0,\\
		\label{eq:simp_MPCC_G_perp}
		G(x)  \mu &=0,\\
		\label{eq:simp_MPCC_H_perp}
		H(x)  \nu &=0, \\
		(\mu<0\land\nu<0)\lor
		\mu\nu&=0.
		\label{eq:simp_MPCC_biactive}
	\end{align}
\end{subequations}
We want to find a function $\phi\colon\R^4\to\R^k$
such that \eqref{eq:simp_MPCC} can be equivalently rewritten as
$F(x,\mu,\nu)=0$ where $F:\R^{n+1+1}\to\R^{n+k}$ has the form
\begin{equation*}
	F(x,\mu,\nu) :=
	\begin{bmatrix}
		\nabla_x\LL(x,\mu,\nu) \\
		\phi(G(x),H(x),\mu,\nu)
	\end{bmatrix}.
\end{equation*}
Since we want to apply a Newton method, we require $k=2$.
We also need that the associated Newton matrices $DF(\cdot)$ are invertible
in a neighborhood of the solution of the system $F(x,\mu,\nu)=0$.
This invertibility will be guaranteed by
certain properties of the Newton derivative $D\varphi$.

\subsection{A nonlinear M-stationarity function}

Recall that the set $M$ from \eqref{eq:M_stationarity_set}
corresponds to the M-stationarity conditions \eqref{eq:simp_MPCC}.
We define $\psi_1,\psi_2,\psi_3,\varphi_1\colon\R^4\to\R$ via
\begin{subequations}
	\label{eq:psi_def}
	\begin{align}
		\label{eq:psi_def_1}
		\psi_1(a,b,\mu,\nu) &:= \max(-a,\abs{b},\abs\mu),
		\\
		\label{eq:psi_def_2}
		\psi_2(a,b,\mu,\nu) &:= \max(-b,\abs{a},\abs\nu),
		\\
		\label{eq:psi_def_3}
		\psi_3(a,b,\mu,\nu) &:= \max(\abs{a},\abs{b},\mu,\nu),
		\\
		\label{eq:phi_def_1}
		\phi_1(a,b,\mu,\nu) &:= \min_{i=1,2,3} \psi_i(a,b,\mu,\nu).
	\end{align}
\end{subequations}
Next, we show that
$M$ is precisely the zero level set of $\varphi_1$.
To this end, we note that
\begin{align*}
	M
	=
	\set{(a,b,\mu,\nu) \in \R^4 \given a \ge 0, \; b = \mu = 0}
	\cup
	\set{(a,b,\mu,\nu) \in \R^4 \given b \ge 0, \; a = \nu = 0}&
	\\
	{}\cup
	\set{(a,b,\mu,\nu) \in \R^4 \given a = b = 0, \mu \le 0, \nu \le 0}&
	,
\end{align*}
i.e., $M$ can be written as the
union of three convex, closed sets.
\begin{lemma}
	\label{lem:M_is_roots}
	Let $(a,b,\mu,\nu) \in \R^4$ be given.
	Then $(a,b,\mu,\nu) \in M$ holds
	if and only if $\varphi_1(a,b,\mu,\nu) = 0$ is valid.
\end{lemma}
\begin{proof}
	It is clear that $\varphi_1(a,b,\mu,\nu) \ge 0$.
	Hence, $\varphi_1(a,b,\mu,\nu) = 0$
	if and only if
	one of the functions $\psi_1$, $\psi_2$, and $\psi_3$
	vanishes at $(a,b,\mu,\nu)$.

	Now, we observe the equivalencies
	\begin{align*}
		\psi_1(a,b,\mu,\nu) = 0
		&\quad\Leftrightarrow\quad
		-a \le 0,\; b = \mu = 0
		,
		\\
		\psi_2(a,b,\mu,\nu) = 0
		&\quad\Leftrightarrow\quad
		-b \le 0,\; a = \nu = 0
		,
		\\
		\psi_3(a,b,\mu,\nu) = 0
		&\quad\Leftrightarrow\quad
		a = b = 0,\; \mu,\nu \le 0
		.
	\end{align*}
	Thus, $\varphi_1(a,b,\mu,\nu) = 0$ holds
	if and only if one of the left hand sides is true
	and
	$(a,b,\mu,\nu) \in M$
	if and only if one of the right hand sides is true.
	This shows the claim.
\end{proof}
We remark that $\phi_1$ describes the distance of a point $(a,b,\mu,\nu)$
to the set $M$ in the $\ell^\infty$-norm.
This follows from the above representation of $M$
and the fact that each of $\psi_1$, $\psi_2$, $\psi_3$
is the distance to one of the three convex subsets of $M$.

We choose the Newton derivative $D\phi_1$ of $\phi_1$
according to the conventions established in \eqref{eq:newton_derivative_min_max}
and \eqref{eq:newton_derivative_absolute_value}
together with the application of the chain rule.
This choice for $D\phi_1$ is fixed for the remainder of the article.
It implies that
\begin{equation}
	\label{eq:d_phi_1_only_unit_vectors}
	D\phi_1(a,b,\mu,\nu) \in \left\{\pm e_1^\top,\pm e_2^\top,\pm e_3^\top,\pm e_4^\top\right\}
\end{equation}
holds for all $(a,b,\mu,\nu)\in\R^4$.

Let us introduce the other component $\phi_2\colon\R^4\to\R$, which is defined via
\begin{equation}
	\label{eq:def_phi_2}
	\phi_2(a,b,\mu,\nu) :=
	\begin{cases}
		\min(\abs{b},\abs\nu) &\text{if}\; 
		D\phi_1(a,b,\mu,\nu) = \pm e_1^\top,
		\\
		\min(\abs{a},\abs\mu) &\text{if}\;
		D\phi_1(a,b,\mu,\nu) = \pm e_2^\top,
		\\
		\abs{b} &\text{if}\;
		D\phi_1(a,b,\mu,\nu) = \pm e_3^\top,
		\\
		\abs{a} &\text{if}\;
		D\phi_1(a,b,\mu,\nu) = \pm e_4^\top,
	\end{cases}
\end{equation}
where the cases are exhaustive due to \eqref{eq:d_phi_1_only_unit_vectors}.
For the prospective Newton derivative $D\phi_2$ of $\phi_2$
we again use the convention established in \eqref{eq:newton_derivative_min_max} 
as well as \eqref{eq:newton_derivative_absolute_value}
and use the same distinction of cases as in \eqref{eq:def_phi_2}.
The Newton differentiability of $\phi_2$ will be shown in \cref{lem:phi_properties} below.
Finally, let $\phi\colon\R^4\to\R^2$ be the function with components $\phi_1$ and $\phi_2$.
The rows of the Newton derivative $D\phi$ of $\phi$ are given by $D\phi_1$ and $D\phi_2$.
Motivated by our arguments from \cref{sec:preliminaries_NMSt_function}, we call $\varphi$
a nonlinear M-stationarity (NMS) function, see \cref{lem:phi_properties} below as well.

In the next lemma, we will look at the possible values of
the Newton derivative of $\phi$ at points from $M$.
This will be an important result in order to show that the Newton
matrix $DF(\cdot)$ is invertible in a neighborhood of $M$,
see \cref{thm:uniform_invertibility}.
If $\varphi_2$ is chosen differently,
one might obtain less tight estimates for the Newton matrices
$D\varphi$,
and this would result in more restrictive
assumptions for the semismooth Newton method below,
cf.\ the proof of \cref{thm:uniform_invertibility}.
Furthermore, we would like to point the reader's attention to the
fact that the upcoming result is based on the precise conventions from
\eqref{eq:newton_derivative_min_max} and \eqref{eq:newton_derivative_absolute_value}
for the Newton derivative of maximum, minimum, and absolute value as well
as the chain rule from \cref{lem:chain_rule}. More precisely, an alternative
strategy for the choice of the Newton derivatives in 
\cref{ex:Newton_differentiability_of_min_max,ex:Newton_differentiability_of_absolute_value}
is likely to cause the next lemma to be false.
For brevity of notation, we define
the sets of matrices
\begin{equation*}
	\forall i,j\in\{1,\ldots,4\}\colon\quad
	J_{i,j} := \set*{
		\begin{bmatrix}
			\pm e_i^\top
			\\
			\pm e_j^\top
		\end{bmatrix}
	} \cup\set*{
		\begin{bmatrix}
			\pm e_j^\top
			\\
			\pm e_i^\top
		\end{bmatrix}
	}
	\subset \R^{2\times 4}.
\end{equation*}
\begin{lemma}
	\label{lem:nice_jacobians}
	Let $\bar w = (\bar a, \bar b, \bar \mu, \bar \nu) \in M$ be given.
	Then there exists $\varepsilon > 0$ such that
	for all $w = (a, b, \mu, \nu) \in B_\varepsilon(\bar w)$, we have
	\begin{subequations}
		\label{eq:nice_jacbians}
		\begin{align}
			\label{eq:nice_jacbians_1}
			\bar a > 0 &\quad\Rightarrow\quad D\varphi(w) \in J_{2,3}, \\
			\label{eq:nice_jacbians_2}
			\bar b > 0 &\quad\Rightarrow\quad D\varphi(w) \in J_{1,4}, \\
			\label{eq:nice_jacbians_3}
			\bar\mu \ne 0 &\quad\Rightarrow\quad D\varphi(w) \in J_{1,2} \cup J_{1,4}, \\
			\label{eq:nice_jacbians_4}
			\bar\nu \ne 0 &\quad\Rightarrow\quad D\varphi(w) \in J_{1,2} \cup J_{2,3}, \\
			\label{eq:nice_jacobians_5}
			\bar w = 0 &\quad\Rightarrow\quad 
			D\phi(w)\in J_{1,2}\cup J_{2,3}\cup J_{1,4}.
		\end{align}
	\end{subequations}
\end{lemma}
\begin{proof}
	Due to the definition of $\phi_2$,
	the possible values for $D\phi(w)$ can only be in
	$J_{1,2}\cup J_{2,3}\cup J_{1,4}$ for all $w\in\R^4$.
	Clearly, the implication \eqref{eq:nice_jacobians_5} follows immediately.

	Suppose that $\bar a>0$ holds.
	Then we have $\bar b=\bar\mu=0$.
	Therefore, there exists $\eps>0$ such that $\max(\abs{b},\abs\mu)<a$ holds
	for all $w=(a,b,\mu,\nu)\in B_\eps(\bar w)$.
	It follows that $\phi_1=\psi_1 < \min(\psi_2,\psi_3)$ holds on $B_\eps(\bar w)$.
	Thus we obtain $D\phi_1(w)\in\set{\pm e_2^\top,\pm e_3^\top}$.
	If we again consider that $\abs\mu<\abs{a}$ then 
	the implication \eqref{eq:nice_jacbians_1} follows.
	The implication \eqref{eq:nice_jacbians_2} can be shown in a similar way.

	Let us consider the case $\bar\mu\neq0$ and $\bar\nu\neq0$.
	Then we have $\bar a=\bar b=0$
	and also $\bar\mu,\bar\nu < 0$.
	Therefore, there exists $\eps>0$ such that 
	$\max(\abs{a},\abs{b})<\min(-\mu,-\nu)$ holds
	for all $w=(a,b,\mu,\nu)\in B_\eps(\bar w)$.
	It follows that $\phi_1=\psi_3 < \min(\psi_1, \psi_2)$ holds on $B_\eps(\bar w)$.
	Thus we obtain $D\phi_1(w)\in\set{\pm e_1^\top,\pm e_2^\top}$.
	If we consider that $\abs{a}<\abs{\mu}$ and $\abs{b}<\abs{\nu}$ then 
	$D\phi(w)\in J_{1,2}$ follows.

	Next, we consider the case that $\bar\mu\neq0$ but $\bar\nu=\bar b=0$.
	Then we have $\bar a=0$.
	Therefore, there exists $\eps>0$ such that 
	$\max(\abs{a},\abs{b},\abs\nu)<\abs\mu$ holds
	for all points $w=(a,b,\mu,\nu)\in B_\eps(\bar w)$.
	It follows that $\phi_1 < \psi_1$ holds on $B_\eps(\bar w)$.
	By a distinction of cases we can obtain that 
	$D\phi_1(w)\in\set{\pm e_1^\top,\pm e_2^\top,\pm e_4^\top}$.
	If we consider \eqref{eq:def_phi_2} and that $\abs{a}<\abs\mu$ then
	$D\phi(w)\in J_{1,2}\cup J_{1,4}$ follows.

	For the case that $\bar\mu\neq0$, $\bar\nu=0$, but $\bar b>0$
	we already know from \eqref{eq:nice_jacbians_2} that 
	$D\phi(w)\in J_{1,2}\cup J_{1,4}$ holds as well.
	If we combine the previous cases, then we obtain \eqref{eq:nice_jacbians_3}.
	The implication \eqref{eq:nice_jacbians_4} can be shown in a similar way.
\end{proof}
We continue with some notable properties of $\phi$.
The first property is important because it allows us
to characterize M-stationarity points as the solution set of
a (nonsmooth) equation, and this is the essential
property of an NMS-function.
\begin{lemma}
	\label{lem:phi_properties}
		\skipline
	\begin{enumerate}
		\item
			\label{item:phi_characterizes_m_stat}
			We have $\phi(a,b,\mu,\nu)=0$ if and only if $(a,b,\mu,\nu)\in M$.
		\item
			\label{item:phi_newton_differentiable}
			The function $\phi$ is Newton differentiable of order $\infty$ on $M$.
		\item
			\label{item:phi_not_cont}
			The function $\phi$ is not continuous
			in any open neighborhood of $M$.
		\item
			\label{item:phi_cont}
			The function $\phi$ is calm at every point $\bar w=(\bar a,\bar b,\bar \mu,\bar \nu)\in M$
			with calmness modulus $1$, i.e., there is a neighborhood $U$ of $\bar w$
			such that
			\[
				\forall w\in U\colon\quad
				\norm{\phi(w)-\phi(\bar w)}\leq\norm{w-\bar w}.
			\]
		\item
			\label{item:phi_cont_2}
			If the sequence $(w_k)_{k \in \N} \subset \R^4$ satisfies $\varphi(w_k) \to 0$
			and $w_k \to \bar w \in \R^4$,
			then $\varphi(\bar w) = 0$.
	\end{enumerate}
\end{lemma}
\begin{proof}
	We start with part~\ref{item:phi_characterizes_m_stat}.
	\cref{lem:M_is_roots} shows
	$(a,b,\mu,\nu)\in M$
	if and only if $\phi_1(a,b,\mu,\nu)=0$.
	Thus, it remains to show that $\phi_2(a,b,\mu,\nu)=0$
	for all points $(a,b,\mu,\nu)\in M$.
	Let $(a,b,\mu,\nu) \in M$ be given.
	We consider the case that $a>0$.
	Then $b=\mu=0$ follows.
	Due to \eqref{eq:nice_jacbians_1} we have
	$D\phi_1(a,b,\mu,\nu)\in\set{\pm e_2^\top,\pm e_3^\top}$,
	which implies $\phi_2(a,b,\mu,\nu)=0$.
	For the case that $b>0$ we can argue similarly.
	For the remaining case $a=b=0$ the property $\phi_2(a,b,\mu,\nu)=0$
	follows directly from the definition of $\phi_2$.

	For part~\ref{item:phi_newton_differentiable},
	let us fix a point $\bar w=(\bar a,\bar b,\bar \mu,\bar \nu)\in M$.
	For $\phi_1$, the Newton differentiability of order $\infty$
	follows from the chain rule \cref{lem:chain_rule}.
	Due to $\phi_2(\bar w)=0$, it suffices to show that
	\begin{equation}
		\label{eq:phi_2_newton_equality}
		\phi_2(w)= D\phi_2(w)(w-\bar w)
	\end{equation}
	holds in a neighborhood of $\bar w$.
	Let $\varepsilon > 0$ from \cref{lem:nice_jacobians} be given
	and consider $w = (a,b,\mu,\nu) \in B_\varepsilon(\bar w)$.
	In case $D\phi_2(w) = \pm e_1^\top$, \eqref{eq:nice_jacbians_1}
	implies $\bar a = 0$ and from the definition of $D\varphi_2$,
	we get $\varphi_2(w) = \pm a$.
	Hence, \eqref{eq:phi_2_newton_equality} follows.
	In case $D\phi_2(w) = \pm e_3^\top$, \eqref{eq:nice_jacbians_3}
	implies $\bar\mu = 0$ and from the definition of $D\varphi_2$,
	we get $\varphi_2(w) = \pm \mu$.
	Again, \eqref{eq:phi_2_newton_equality} follows.
	The remaining cases follow analogously.
	
	We continue with part~\ref{item:phi_not_cont}.
	Any open neighborhood of $M$ contains the point
	$w_t:=(2t,2t,t,0)$ for some $t > 0$.
	It can be shown that
	\begin{equation*}
		\phi_2(2t,2t,t,0) = t
		\neq
		0=\lim_{s\downto0}
		\phi_2(2t-s,2t,t,0)
	\end{equation*}
	holds.
	Hence $\phi_2$ is not continuous at $w_t$.

	In order to show part~\ref{item:phi_cont},
	we can utilize part~\ref{item:phi_newton_differentiable},
	which implies that
	\begin{equation*}
		\phi(w)-\phi(\bar w)=D\phi(w)(w-\bar w)
	\end{equation*}
	holds for all $w$ in a neighborhood of $\bar w$.
	Since $\norm{D\phi(w)}=1$ holds in a neighborhood of $\bar w$
	due to \cref{lem:nice_jacobians},
	we get
	$ \norm{\phi( w)-\phi(\bar w)} \leq \norm{D\phi(w)}\norm{w-\bar w} = \norm{ w -\bar w}$.

	For part \ref{item:phi_cont_2},
	we first use the continuity of $\varphi_1$
	to obtain $\varphi_1(\bar w) = 0$,
	i.e., $\bar w \in M$, see \cref{lem:M_is_roots}.
	From part \ref{item:phi_characterizes_m_stat}
	we conclude $\varphi(\bar w) = 0$.
\end{proof}

The following lemma will be useful in order to interpret the semismooth Newton method
as an active set strategy for \eqref{eq:MPCC}.
\begin{lemma}
	\label{lem:for_interpretation_as_AS}
	Let $w = (a,b,\mu,\nu)$ and $\delta w = (\delta a, \delta b, \delta\mu, \delta\nu)$
	be given.
	Then we have the equivalence
	\begin{equation*}
		D\varphi(w) \, \delta w = -\varphi(w)
		\quad\Leftrightarrow\quad
		\begin{cases}
			\delta b = - b,\, \delta \mu = - \mu
			&\quad\text{if}\; D\varphi(w)\in J_{2,3},\\
			\delta a = - a,\, \delta \nu = - \nu
			&\quad\text{if}\; D\varphi(w)\in J_{1,4},\\
			\delta a = - a,\, \delta b = - b
			&\quad\text{if}\; D\varphi(w)\in J_{1,2}.
		\end{cases}
	\end{equation*}
\end{lemma}
\begin{proof}
	We first consider
	$D\varphi(w)\in J_{2,3}$.
	% The remaining cases can be shown by using the very same arguments.
	The set $J_{2,3}$ contains exactly the eight matrices
	\begin{align*}
		&
		\begin{pmatrix}
			0 & 1 & 0 & 0
			\\
			0 & 0 & 1 & 0
		\end{pmatrix}
		,\;
		\begin{pmatrix}
			0 & 1 & 0 & 0
			\\
			0 & 0 & -1 & 0
		\end{pmatrix}
		,\;
		\begin{pmatrix}
			0 & -1 & 0 & 0
			\\
			0 & 0 & 1 & 0
		\end{pmatrix}
		,\;
		\begin{pmatrix}
			0 & -1 & 0 & 0
			\\
			0 & 0 & -1 & 0
		\end{pmatrix}
		,\\
		&
		\begin{pmatrix}
			0 & 0 & 1 & 0
			\\
			0 & 1 & 0 & 0
		\end{pmatrix}
		,\;
		\begin{pmatrix}
			0 & 0 & -1 & 0
			\\
			0 & 1 & 0 & 0
		\end{pmatrix}
		,\;
		\begin{pmatrix}
			0 & 0 & 1 & 0
			\\
			0 & -1 & 0 & 0
		\end{pmatrix}
		,\;
		\begin{pmatrix}
			0 & 0 & -1 & 0
			\\
			0 & -1 & 0 & 0
		\end{pmatrix}
		.
	\end{align*}
	We discuss the case that $D\varphi(w)$ coincides with the second matrix,
	i.e., $D\varphi_1(w) = e_2^\top$ and $D\varphi_2(w) = -e_3^\top$.
	The crucial point of this proof is the following:
	Since the function $\varphi_1$ is composed by a composition of $\min$ and $\max$,
	we can utilize
	$D\varphi_1(w) = e_2^\top$
	to obtain
	$\varphi_1(w) = b$
	(recall that $D\varphi_1$ is chosen according to
	the convention established in \cref{ex:Newton_differentiability_of_min_max}
	and the chain rule).
	Hence, we find
	\begin{equation*}
		D\varphi_1(w) \, \delta w = -\varphi_1(w)
		\quad\Leftrightarrow\quad
		e_2^\top \delta w = -b
		\quad\Leftrightarrow\quad
		\delta b = -b.
	\end{equation*}
	Similarly, from $D\varphi_2(w) = -e_3^\top$ we obtain $\varphi_2(w) = -\mu$
	and, thus,
	\begin{equation*}
		D\varphi_2(w) \, \delta w = -\varphi_2(w)
		\quad\Leftrightarrow\quad
		-e_3^\top \delta w = \mu
		\quad\Leftrightarrow\quad
		-\delta\mu=\mu.
	\end{equation*}
	This finishes the proof for this particular case.
	The remaining $23$ cases follow similarly.
\end{proof}

\section{Application of a semismooth Newton method}
\label{sec:newton_method}

Using the NCP-function $\pi_\textup{min}\colon\R^2\to\R$ as well as the NMS-function
$\varphi\colon\R^4\to\R^2$ constructed in \cref{sec:M_St_as_system}, we introduce
$F\colon\R^n\times\R^\ell\times\R^m\times\R^p\times\R^p\to \R^n\times\R^\ell\times\R^m\times\R^{2p}$
via
\begin{equation}\label{eq:def_F}
	F(x,\lambda,\eta,\mu,\nu)
	:=
	\begin{bmatrix}
		\nabla_x\LL(x,\lambda,\eta,\mu,\nu)
		\\
		[
			\pi_\textup{min}( -g_i(x), \lambda_i ) 
		]_{I^\ell}
		\\
		h(x)
		\\
		[
			\varphi( G_i(x), H_i(x), \mu_i, \nu_i )
		]_{I^p}
	\end{bmatrix}.
\end{equation}
Clearly, by \cref{lem:phi_properties}, a point $x\in\R^n$ is M-stationary for \eqref{eq:MPCC}
if and only if there is a quadruple $(\lambda,\eta,\mu,\nu)$ such that $F(x,\lambda,\eta,\mu,\nu)=0$
holds. In this case, it holds $(\lambda,\eta,\mu,\nu)\in\Lambda^\textup{M}(x)$.
Observing that all the data functions $f$, $g$, $h$, $G$, and $H$ are twice continuously 
differentiable, \cref{lem:chain_rule,lem:phi_properties} guarantee that $F$ is Newton differentiable
on the set of its roots.
Thus, we may apply the semismooth Newton method from \cref{sec:nonsmooth_Newton} in order to
find the roots of $F$, i.e., M-stationary points of \eqref{eq:MPCC}.

In order to guarantee convergence of the Newton method to an M-stationary point $x\in\R^n$ of \eqref{eq:MPCC}
with associated multiplier $(\lambda,\eta,\mu,\nu)\in\Lambda^\textup{M}(x)$, we have to guarantee
that the Newton derivative of $F$ is uniformly invertible in a neighborhood of $z:=(x,\lambda,\eta,\mu,\nu)$.
Abstractly, we have
\begin{align*}
	DF(z)
	=
	\begin{bmatrix}
		\nabla^2_{xx}\LL(z)
		&
		g'(x)^\top
		&
		h'(x)^\top
		&
		G'(x)^\top
		&
		H'(x)^\top
		\\
		A_1(z)
		&
		A_2(z)
		&
		0
		&
		0
		&
		0
		\\
		h'(x)
		&
		0
		&
		0
		&
		0
		&
		0
		\\
		A_3(z)
		&
		0
		&
		0
		&
		A_4(z)
		&
		A_5(z)
	\end{bmatrix}
\end{align*}
for the Newton derivative of $F$ at $z$ were we used
\begin{align*}
	A_1(z)&:=[-D_a\pi_\textup{min}(-g_i(x),\lambda_i)\nabla g_i(x)^\top]_{I^\ell},\\
	A_2(z)&:=[D_b\pi_\textup{min}(-g_i(x),\lambda_i)e_i^\top]_{I^\ell},\\
	A_3(z)&:=[D_a\varphi(G_i(x), H_i(x), \mu_i, \nu_i)\nabla G_i(x)^\top
	    		+D_b\varphi(G_i(x), H_i(x), \mu_i, \nu_i)\nabla H_i(x)^\top]_{I^p},\\
	A_4(z)&:=[D_\mu\varphi(G_i(x), H_i(x), \mu_i, \nu_i)e_i^\top]_{I^p},\\
	A_5(z)&:=[D_\nu\varphi(G_i(x), H_i(x), \mu_i, \nu_i)e_i^\top]_{I^p}.
\end{align*}

\begin{theorem}
	\label{thm:uniform_invertibility}
	Let $\bar x\in\R^n$ be an M-stationary point of \eqref{eq:MPCC} where
	MPCC-LICQ holds. Furthermore, assume that MPCC-SSOC holds at $\bar x$ w.r.t.\ the multipliers
	$(\bar\lambda,\bar\eta,\bar\mu,\bar\nu)\in\Lambda^\textup{M}(\bar x)$.
	Set $\bar z:=(\bar x,\bar\lambda,\bar\eta,\bar\mu,\bar\nu)$ and observe that this
	point solves $F(\bar z)=0$.
	Then there exist $\varepsilon > 0$ and $C > 0$ such that
	$\norm{DF(z)^{-1}} \le C$ for all $z \in B_\varepsilon(\bar z)$.
\end{theorem}
\begin{proof}
	First, we provide a result
	for a linear system associated with the solution $\bar z$.
	To this end, let matrices
	% $P_i \in \{(1,0), (0,1)\}$, $i \in I^\ell$,
	% $Q_j \in J_{1,2}\cup J_{2,3}\cup J_{1,4}$, $j \in I^p$,
	$P_i\in \R^{1\times 2}$, $Q_j\in \R^{2\times 4}$ 
	for
	$i \in I^\ell$, $j \in I^p$
	be given such that
	\begin{equation}
		\label{eq:matrix_conditions}
		\left\{
			\begin{aligned}
				P_i 
					&{}\in \bigl\{	\begin{pmatrix}1&0\end{pmatrix},
											\begin{pmatrix}0&1\end{pmatrix}\bigr\} &
				Q_j &{}\in J_{1,2}\cup J_{2,3}\cup J_{1,4} \\
				g_i(\bar x) < 0   &\quad\Rightarrow\quad P_i = \begin{pmatrix}0&1\end{pmatrix}, &
				\bar\lambda_i > 0 &\quad\Rightarrow\quad P_i = \begin{pmatrix}1&0\end{pmatrix}, \\
				G_j(\bar x) > 0   &\quad\Rightarrow\quad Q_j \in J_{2,3}, &
				H_j(\bar x) > 0   &\quad\Rightarrow\quad Q_j \in J_{1,4}, \\
				\bar\mu_j \ne 0   &\quad\Rightarrow\quad Q_j \in J_{1,2} \cup J_{1,4}, &
				\bar\nu_j \ne 0   &\quad\Rightarrow\quad Q_j \in J_{1,2} \cup J_{2,3}
				\mspace{-32mu}
			\end{aligned}
			\qquad
		\right.
	\end{equation}
	holds for all $i \in I^\ell$, $j \in I^p$,
	cf.\ \eqref{eq:nice_jacbians}.
	Associated with this choice of $P_i$, $Q_j$, we define the index sets
	\begin{align*}
		I^\ell_1 &:= \left\{i \in I^\ell \,\middle|\,P_i = \begin{pmatrix}1&0\end{pmatrix}\right\}
		,&
		I^\ell_2 &:= \left\{i \in I^\ell\,\middle|\,P_i = \begin{pmatrix}0&1\end{pmatrix}\right\}
		,
		\\
		I^p_{1,2} &:= \set{j \in I^p \given Q_j \in J_{1,2}}
		,&
		I^p_{1,4} &:= \set{j \in I^p \given Q_j \in J_{1,4}}
		,\\
		I^p_{2,3} &:= \set{j \in I^p \given Q_j \in J_{2,3}}
		.
	\end{align*}
	Now, 
	we consider the linear system
	with unknowns $\delta z = (\delta x, \delta\lambda, \delta\eta, \delta\mu, \delta\nu)$
	\begin{subequations}
		\label{eq:a_system_too_large_for_me}
		\begin{align}
			\nabla_{xx}^2 \LL(\bar z) \, \delta x
			+
			g'(\bar x)^\top \, \delta\lambda
			+
			h'(\bar x)^\top \, \delta\eta&
			\notag
			\\
			+
			G'(\bar x)^\top \, \delta\mu
			+
			H'(\bar x)^\top \, \delta\nu
			&=
			r,
			\label{eq:a_system_too_large_for_me_1}
			\\
			P_i \,
			\begin{pmatrix}
				-\nabla g_i(\bar x)^\top \delta x\\
				\delta\lambda_i
			\end{pmatrix}
			&=
			s_i
			\quad\forall i \in I^\ell,
			\label{eq:a_system_too_large_for_me_2}
			\\
			h'(\bar x) \, \delta x &=
			t,
			\label{eq:a_system_too_large_for_me_3}
			\\
			Q_j \,
			\begin{pmatrix}
				\nabla G_j(\bar x)^\top \delta x\\
				\nabla H_j(\bar x)^\top \delta x\\
				\delta \mu_j \\
				\delta \nu_j
			\end{pmatrix}
			&=
			\begin{pmatrix}
				u_j \\ v_j
			\end{pmatrix}
			\quad \forall j \in I^p
			,
			\label{eq:a_system_too_large_for_me_4}
		\end{align}
	\end{subequations}
	where $r\in\R^n$, $s\in\R^\ell$, $t\in\R^m$, and $u,v\in\R^p$
	form a given right-hand side.
	Let us inspect the block \eqref{eq:a_system_too_large_for_me_2}.
	In case $i \in I^\ell_2$, this block is equivalent to
	$\delta\lambda_i = s_i$.
	Hence, we can eliminate these variables.
	Now, we consider the last block \eqref{eq:a_system_too_large_for_me_4}.
	In case $j \in I^p_{1,2}$, i.e., $Q_j \in J_{1,2}$, we can assume w.l.o.g.\ that
	\begin{equation*}
		Q_j =
		\begin{pmatrix}
			1 & 0 & 0 & 0 \\
			0 & 1 & 0 & 0
		\end{pmatrix}
		.
	\end{equation*}
	Hence, the $j$th component of the last block is equivalent to
	\begin{equation*}
		\nabla G_j(\bar x)^\top \delta x = u_j
		\quad\text{and}\quad
		\nabla H_j(\bar x)^\top \delta x = v_j
		.
	\end{equation*}
	For $j \in I^p_{1,4}$, the last block is w.l.o.g.\ equivalent to
	\begin{equation*}
		\nabla G_j(\bar x)^\top \delta x = u_j
		\quad\text{and}\quad
		\delta\nu_j = v_j
	\end{equation*}
	and for $j \in I^p_{2,3}$ we get
	\begin{equation*}
		\nabla H_j(\bar x)^\top \delta x = v_j
		\quad\text{and}\quad
		\delta\mu_j = u_j
		.
	\end{equation*}
	Thus, the values $\delta\nu_j$ for $j \in I^p_{1,4}$
	and $\delta\mu_j$ for $j \in I^p_{2,3}$ can be eliminated in the above system.
	With the index sets
	\begin{equation*}
		I^p_\mu := I^p \setminus I^p_{2,3}
		,\qquad
		I^p_\nu := I^p \setminus I^p_{1,4}
	\end{equation*}
	we arrive at the reduced saddle-point system
	\begin{align*}
		\nabla_{xx}^2 \LL(\bar z) \, \delta x
		+
		g'(\bar x)^\top_{I^\ell_1} \, \delta\lambda_{I^\ell_1}
		+
		h'(\bar x)^\top \, \delta\eta
		+
		G'(\bar x)^\top_{I^p_\mu} \, \delta\mu_{I^p_\mu}
		+
		H'(\bar x)^\top_{I^p_\nu} \, \delta\nu_{I^p_\nu}
		&=
		\tilde r,
		\\
		-g'(\bar x)_{I^\ell_1} \, \delta x &= s_{I^\ell_1},
		\\
		h'(\bar x) \, \delta x &=
		t,
		\\
		G'(\bar x)_{I^p_\mu} \delta x &= u_{I^p_\mu},
		\\
		H'(\bar x)_{I^p_\nu} \delta x &= u_{I^p_\nu}
		.
	\end{align*}
	Therein, the modified right-hand side $\tilde r$
	results from the elimination of some of the multipliers.
	It
	can be bounded by
	$r$, $s$, $u$, and $v$.
	Note that this reduced system is symmetric.
	Furthermore, it clearly holds
	\begin{subequations}
		\label{eq:properties_index_sets}
		\begin{align}
			\{i\in I^g(\bar x)\,|\,\bar\lambda_i>0\}
			\subset
			I^\ell_1&\subset I^g(\bar x),
			\\
			I^{0+}(\bar x)\cup I^{00}_{\pm\R}(\bar x,\bar\mu,\bar\nu)
			\subset
			I^p_{\mu}
			&
			\subset I^{0+}(\bar x)\cup I^{00}(\bar x),
			\\
			I^{+0}(\bar x)\cup\ I^{00}_{\R\pm}(\bar x,\bar\mu,\bar\nu)
			\subset
			I^p_{\nu}
			&
			\subset
			I^{+0}(\bar x)\cup I^{00}(\bar x)
		\end{align}
	\end{subequations}
	by definition of these index sets.
	Additionally, we have $I^p_\mu\cup I^p_\nu=I^p$
	due to \cref{lem:nice_jacobians}.
	By MPCC-LICQ, the constraint block of the reduced system is surjective
	and
	from MPCC-SSOC we get that
	the matrix $\nabla_{xx}^2\LL(\bar z)$ is positive definite
	on the kernel of the constraint block.
	Now, \cref{lem:saddle_point_system}
	implies the invertibility of the system.
	By undoing the elimination of some of the multipliers,
	we find that the system \eqref{eq:a_system_too_large_for_me} is
	invertible, i.e.,
	there is a constant $c > 0$, such that the unique solution $\delta z$ of
	\eqref{eq:a_system_too_large_for_me}
	satisfies
	$\norm{\delta z} \le c \, \paren{\norm{r} + \norm{s} + \norm{t} + \norm{u} + \norm{v}}$.
	Note that the constant $c$ might depend on the matrices $P_i$ and $Q_j$.
	However, since there are only finitely many choices
	for the matrices $P_i$ and $Q_j$,
	we can choose $c$ large enough,
	such that this estimate holds for all
	values of $P_i$ and $Q_j$ satisfying \eqref{eq:matrix_conditions}.

		At this point of the proof, we have shown the following:
		There is a constant $c$, such that the linear system
		\eqref{eq:a_system_too_large_for_me} is uniformly invertible
		for any choice of matrices $P_i$, $Q_j$ satisfying \eqref{eq:matrix_conditions}.

	It remains to prove the uniform invertibility of the Newton matrix $DF(z)$
	for all $z$ in a neighborhood of $\bar z$.
	First, we can utilize \cref{lem:nice_jacobians}
	and the continuity of $g$, $G$, and $H$
	to obtain $\varepsilon > 0$ such that
	$P_i := D\pi_\textup{min}( -g_i(x), \lambda_i)$
	and
	$Q_j := D\varphi( G_j(x), H_j(x), \mu_j, \nu_j)$
	satisfy \eqref{eq:matrix_conditions}
	for all $i \in I^\ell$, $j \in I^p$,
	and $z\in B_\eps(\bar z)$.
	Note that we still use $\bar x$, $\bar\lambda$, $\bar\mu$, and $\bar\nu$ in \eqref{eq:matrix_conditions}.
	Thus, the Newton matrix $DF(z)$
	is a perturbation of the system matrix from \eqref{eq:a_system_too_large_for_me}
	for this particular choice of the matrices $P_i$ and $Q_j$.
	Since $f$, $g$, $h$, $G$, and $H$ are assumed to be twice continuously differentiable,
	the perturbation can be made arbitrarily small (by reducing $\varepsilon$ if necessary).
	Thus, \cref{lem:neumann_series}
	ensures that we get a uniform bound for
	$DF(z)^{-1}$ for all $z \in B_\varepsilon(\bar z)$.
\end{proof}

Now, we are in position to state a local convergence result for our nonsmooth
Newton method based on the map $F$ from \eqref{eq:def_F}. Its proof simply
follows by combining 
\cref{thm:newton_method,thm:uniform_invertibility}.
\begin{theorem}\label{thm:local_convergence}
	Let $\bar x\in\R^n$ be an M-stationary point of \eqref{eq:MPCC}
	where MPCC-LICQ holds.
	Furthermore, assume that MPCC-SSOC holds at $\bar x$ w.r.t.\ the multipliers
	$(\bar \lambda,\bar\eta,\bar\mu,\bar\nu)\in\Lambda^\textup{M}(\bar x)$.
	Set $\bar z:=(\bar x,\bar\lambda,\bar\eta,\bar\mu,\bar\nu)$.
	Then there exists $\delta>0$ such that the nonsmooth Newton method
	from \cref{sec:nonsmooth_Newton}
	applied to the mapping $F$ from \eqref{eq:def_F} is well defined 
	for each initial iterate from $B_\delta(\bar z)$ while the associated
	sequence of iterates converges superlinearly to $\bar z$.
	If, additionally, the second-order derivatives of the data functions
	$f$, $g$, $h$, $G$, and $H$ are locally Lipschitz
	continuous, then the convergence is quadratic.
\end{theorem}

\begin{remark}\label{rem:one_step_convergence_for_linear_quadratic_problems}
	In addition to the assumptions of \cref{thm:local_convergence},
	assume that the cost function $f$ is quadratic while the constraint
	mappings $g$, $h$, $G$, and $H$ are affine in \eqref{eq:MPCC}.
	Then \cref{ex:Newton_differentiability_of_min_max}, \cref{lem:chain_rule},
	and \cref{lem:phi_properties} guarantee that the mapping $F$ from
	\eqref{eq:def_F} is Newton differentiable of order $\infty$ on $M$.
	Thus, \cref{thm:newton_method} guarantees one-step convergence
	for the associated nonsmooth Newton method if the initial iterate
	is sufficiently close to the reference point $\bar z$.
\end{remark}

We note that the example from \cite[Section 7.3]{FletcherLeyfferRalphScholtes2006}
satisfies our assumptions MPCC-LICQ and MPCC-SSOC at the origin which is an
S-stationary point of the underlying complementarity-constrained optimization
problem. Due to \cref{thm:local_convergence}, local superlinear convergence
of our nonsmooth Newton method is guaranteed. 
On the other hand, the SQP-method suggested in \cite{FletcherLeyfferRalphScholtes2006}
only converges linearly to the point of interest.
	Next, we want to compare \cref{thm:local_convergence} with the convergence
	results from \cite{GuoLinYe2015} where a Levenberg--Marquardt method has
	been used to find stationary points of \eqref{eq:MPCC}. 
	In \cite[Theorem~4.2]{GuoLinYe2015}, local quadratic convergence of the method
	has been shown under validity of an abstract error bound condition.
	However, the paper does not present any assumptions in terms of initial problem
	data which ensure the presence of the error bound condition. 
	Even worse, the authors admit in \cite[Section~5]{GuoLinYe2015} that sufficient
	conditions for the validity of this error bound condition, which are well known
	from the literature, are likely to be violated in their special setting.
	Some toy examples are provided where validity of the error bound condition can be shown
	via some reduction arguments.
	In this regard, a qualitative comparison with the method from \cite{GuoLinYe2015}
	is not possible. For a quantitative comparison, we refer the interested reader
	to \cref{sec:numerics}.

We mention that it is possible to
interpret the Newton method as an active set strategy.
To this end, one has to utilize \cref{lem:for_interpretation_as_AS}.
If the current iterate is denoted by
$z_k = (x_k, \lambda_k, \eta_k, \mu_k, \nu_k)$,
the next iterate
$z_{k+1} = (x_{k+1}, \lambda_{k+1}, \eta_{k+1}, \mu_{k+1}, \nu_{k+1})$
solves
the symmetric linear system
\begin{align*}
	\nabla_x \LL(z_k) + \nabla_{xx}^2 \LL(z_k) \, (x_{k+1}-x_k)
	\qquad \qquad \qquad \qquad \qquad
	&\\
	{}+
	g'(x_k)^\top_{I^\ell_1} \, (\lambda_{k+1}-\lambda_k)_{I^\ell_1}
	+
	h'(x_k)^\top \, (\eta_{k+1}-\eta_k)
	\mspace{32mu}
	&\\
	{}+
	G'(x_k)^\top_{I^p_\mu} \, (\mu_{k+1}-\mu_k)_{I^p_\mu}
	+
	H'(x_k)^\top_{I^p_\nu} \, (\nu_{k+1}-\nu_k)_{I^p_\nu}
	&=
	0,
	\\
	g(x_k)_{I^\ell_1} + g'(x_k)_{I^\ell_1} \, (x_{k+1} - x_k)_{I^\ell_1} &= 0,
	&
	(\lambda_{k+1})_{I^\ell \setminus I^\ell_1} &= 0,
	\\
	h(x_k) + h'(x_k) \, (x_{k+1} - x_k) &= 0,
	\\
	G(x_k)_{I^p_\mu} + G'(x_k)_{I^p_\mu} \, (x_{k+1} - x_k)_{I^p_\mu} &= 0,
	&
	(\mu_{k+1})_{I^p \setminus I^p_\mu} &= 0,
	\\
	H(x_k)_{I^p_\nu} + H'(x_k)_{I^p_\nu} \, (x_{k+1} - x_k)_{I^p_\nu} &= 0,
	&
	(\nu_{k+1})_{I^p \setminus I^p_\nu} &= 0
	.
\end{align*}
Here, the index sets
$I^\ell_1$, $I^p_\mu$, and $I^p_\nu$
are constructed similarly as in the proof of \cref{thm:uniform_invertibility}.

Let us briefly compare our algorithm with
\cite[Alg.~2.2]{IzmailovSolodov2008}.
Therein, the authors use an identification procedure
to obtain the active sets $I^{+0}(\bar x)$, $I^{0+}(\bar x)$, and $I^{00}(\bar x)$
and, afterwards, $\bar x$ is approximated by an active set strategy.
This approach is very similar to our suggestion.
For the convergence theory
they presume validity of MPCC-LICQ
and MPCC-SOSC at a given local minimizer of \eqref{eq:MPCC} which, thus,
is an S-stationary point 
(observe that MPCC-SOSC is called \emph{piecewise SOSC} in \cite{IzmailovSolodov2008},
and take notice of \cite[pages~1006-1007]{IzmailovSolodov2008}).
Recall that MPCC-SOSC is slightly weaker than MPCC-SSOC, which is required in our \cref{thm:local_convergence}.
However, the algorithm from \cite{IzmailovSolodov2008}
is designed for the computation of S-stationary points
and cannot approximate M-stationary points (which are not already strongly stationary).

\section{Globalization}
\label{sec:glob}

A possible idea for
the globalization of the nonsmooth Newton method from \cref{sec:nonsmooth_Newton}
is to exploit the squared residual of $F$ as a merit function. 
Unfortunately, it can be easily checked that the resulting map 
$z\mapsto\tfrac12\norm{F(z)}^2$ is not smooth.
Exploiting the well-known fact that the square of the Fischer--Burmeister
function $\pi_\textup{FB}$ is smooth, see, e.g., 
\cite[Proposition~3.4]{FacchineiSoares1997}, we are, however, in position to construct
a smooth merit function.
Therefore, let us first mention that 
the function $\varphi_1$ has the equivalent representation
\begin{subequations}
	\label{eq:theta_def}
	\begin{align}
		\label{eq:theta_def_1}
		\theta_1(a,b,\mu,\nu) &:= \abs{\min(a,b)},
		\\
		\label{eq:theta_def_2}
		\theta_2(a,b,\mu,\nu) &:= \min(\abs a,\abs\mu),
		\\
		\label{eq:theta_def_3}
		\theta_3(a,b,\mu,\nu) &:= \min(\abs b,\abs\nu),
		\\
		\label{eq:theta_def_4}
		\theta_4(a,b,\mu,\nu) &:= \max(0,\min(\mu,\abs\nu),\min(\nu,\abs\mu)),
		\\
		\label{eq:old_phi_def_1}
		\phi_1(a,b,\mu,\nu) &= \max_{i=1,\ldots,4} \theta_i(a,b,\mu,\nu).
	\end{align}
\end{subequations}
Indeed, one can check that the zeros of $\theta(a,b,\mu,\nu):=\max_{i=1,\ldots,4}\theta_i(a,b,\mu,\nu)$
coincide with the set $M$ and, by construction, $\theta$ is $1$-Lipschitz continuous w.r.t.\
the $\ell^\infty$-norm. Based on these observations and an elementary distinction of cases, 
it is now possible to exploit the particular structure of $\theta$ in order to show 
that this function equals the $\ell^\infty$-distance to $M$. 
Noting that $\varphi_1$ from \eqref{eq:phi_def_1} has the same property, $\theta$ and $\varphi_1$
actually need to coincide.
This motivates the definition of
\begin{subequations}
	\label{eq:thetaFB_def_all}
	\begin{align}
		\label{eq:thetaFB_def_1}
		\theta_{1,\textup{FB}}(a,b,\mu,\nu) &:= |\pi_{\textup{FB}}(a,b)|,
		\\
		\label{eq:thetaFB_def_2}
		\theta_{2,\textup{FB}}(a,b,\mu,\nu) &:= \pi_{\textup{FB}}(\abs a,\abs\mu),
		\\
		\label{eq:thetaFB_def_3}
		\theta_{3,\textup{FB}}(a,b,\mu,\nu) &:= \pi_{\textup{FB}}(\abs b,\abs\nu),
		\\
		\label{eq:thetaFB_def_4}
		\theta_{4,\textup{FB}}(a,b,\mu,\nu) &:= 
		\begin{cases}
			0 & \text{if } \mu,\nu \le 0, \\
			\pi_{\textup{FB}}(\abs{\mu},\abs{\nu}) & \text{else},
		\end{cases}
		\\
		\label{eq:thetaFB_def}
		\theta_{\textup{FB}}(a,b,\mu,\nu) &:=
		[
			\theta_{i,\textup{FB}}(a,b,\mu,\nu)
		]_{i = 1,\ldots,4}.
	\end{align}
\end{subequations}
Now, we introduce
$F_{\textup{FB}}\colon\R^n\times\R^\ell\times\R^m\times\R^p\times\R^p\to \R^n\times\R^\ell\times\R^m\times\R^{4p}$,
a modified residual,
as stated below for arbitrary $x\in\R^n$, $\lambda\in\R^\ell$, $\eta\in\R^m$, and $\mu,\nu\in\R^p$:
\begin{equation}
	\label{eq:def_FFB}
	F_{\textup{FB}}(x,\lambda,\eta,\mu,\nu)
	:=
	\begin{bmatrix}
		\nabla_x\LL(x,\lambda,\eta,\mu,\nu)
		\\
		[
			\pi_\textup{FB}( -g_i(x), \lambda_i ) 
		]_{I^\ell}
		\\
		h(x)
		\\
		[
			\theta_{\textup{FB}}( G_i(x), H_i(x), \mu_i, \nu_i )
		]_{I^p}
	\end{bmatrix}.
\end{equation}
In the next lemma, we show that the squared residuals of $F$ and $F_\textup{FB}$ are, in some sense, equivalent.
\begin{lemma}
	\label{lem:equivalence}
	There exist constants $c,C>0$ with
	\begin{equation*}
		c \, \norm{ F_{\textup{FB}}(x,\lambda,\eta,\mu,\nu) }^2
		\le
		\norm{ F(x,\lambda,\eta,\mu,\nu) }^2
		\le
		C \, \norm{ F_{\textup{FB}}(x,\lambda,\eta,\mu,\nu) }^2
	\end{equation*}
	for all $x\in\R^n$, $\lambda\in\R^\ell$, $\eta\in\R^m$, and $\mu,\nu\in\R^p$.
\end{lemma}
\begin{proof}
	Throughout the proof, $w:=(a,b,\mu,\nu)\in\R^4$ is arbitrarily chosen.
	Due to \cite[Lemma~3.1]{Tseng1996}, the functions $\pi_{\textup{min}}$ and $\pi_{\textup{FB}}$
	are equivalent in the sense
	\begin{equation}\label{eq:equivalence_of_min_and_FB}
		\frac{2}{2+\sqrt 2}\,\abs{\pi_\textup{min}(a,b)}
		\leq
		|\pi_\textup{FB}(a,b)|
		\leq
		(2+\sqrt 2)\,\abs{\pi_\textup{min}(a,b)}.
	\end{equation}
	Thus, keeping in mind the definitions of $F$ and $F_\textup{FB}$ from
	\eqref{eq:def_F} and \eqref{eq:def_FFB}, we only need
	to show the equivalence of $\varphi$ and $\theta_\textup{FB}$.
	
	The relation \eqref{eq:equivalence_of_min_and_FB} yields
	\begin{equation*}
		\frac{2}{2 + \sqrt{2}} \,
		\abs{ \theta_i(w) }
		\le
		\abs{ \theta_{i,\textup{FB}}(w) }
		\le
		(2 + \sqrt{2}) \,
		\abs{ \theta_i(w) }
	\end{equation*}
	for $i = 1,\ldots, 4$.
	Together with the estimate
	\begin{equation*}
		\varphi_1(w)
		=
		\max_{i=1,\ldots,4} \theta_i(w)
		\le
		\left(\sum_{i = 1}^4 \theta_i^2(w)\right)^{1/2}
		\le
		2 \, \varphi_1(w),
	\end{equation*}
	we get equivalence of the functions
	$\abs{\varphi_1}$
	and
	$\norm{\theta_{\textup{FB}}}$.

	In order to complete the proof, we only need to show
	\begin{equation}\label{eq:estimate_for_phi_2}
		\abs{\varphi_2(w)}
		\le
		\abs{\varphi_1(w)}
	\end{equation}
	for all $w$
	since this already yields the equivalence of $\phi$ and $\theta_\textup{FB}$.
	Let us distinguish some cases.
	If we have $D\varphi_1(w)\in\{\pm e_1^\top,\pm e_2^\top\}$, then it
	clearly holds 
	\begin{equation*}
		\abs{\varphi_2(w)}
		\le
		% \max\parens[\big]{\min(\abs{a},\abs{\mu}), \min(\abs{b},\abs{\nu})}
		% =
		\max(\theta_2(w),\theta_3(w))
		\leq
		\max_{i=1,\ldots,4} \theta_i(w)
		=
		\varphi_1(w)
		.
	\end{equation*}
	Now, suppose that $D\varphi_1(w)=\pm e_3^\top$ holds. 
	The Newton derivative of $\varphi_1$ is evaluated
	based on the representation of $\varphi_1$ given in \eqref{eq:phi_def_1},
	and thus we obtain the relation $\varphi_1(w) = \min(\psi_1(w),\psi_3(w)) = \abs{\mu}$.
	This implies
	\begin{equation*}
		\varphi_2(w)
		=
		\abs{b}
		\le
		\min( \psi_1(w), \psi_3(w) )
		=
		\varphi_1(w).
	\end{equation*}
	The case $D\varphi_1(w)=\pm e_4^\top$ can be handled analogously.
	This shows \eqref{eq:estimate_for_phi_2} for arbitrary $w$
	and the proof is complete.
\end{proof}

For the globalization of our nonsmooth Newton method, we make use of the merit function 
$\Phi_\textup{FB}\colon\R^n\times\R^\ell\times\R^m\times\R^p\times\R^p\to\R$
given by
\begin{equation*}
	\PhiFB(z)
	:=
	\frac12 \, \norm{F_{\textup{FB}}(z)}^2
\end{equation*}
for all $z=(x,\lambda,\eta,\mu,\nu)\in\R^n\times\R^\ell\times\R^m\times\R^p\times\R^p$.
This function is continuously differentiable:
First, recall that the square of the function $\pi_\textup{FB}$ is
continuously differentiable.
The gradient of $(a,b) \mapsto \pi_{\textup{FB}}(a , b)^2$
vanishes on the complementarity angle $\set{(a,b) \given 0 \le a \perp b \ge 0}$.
This implies the continuous differentiability of
$(a,b) \mapsto \pi_{\textup{FB}}(\abs{a} , \abs{b})^2$.
Similar arguments can be used to check the
continuous differentiability of the function
$\theta_{4,\textup{FB}}^2$.

Now, we can utilize the globalization idea from
\cite[Section~3]{DeLucaFacchineiKanzow2000} and
\cite[Algorithm~3.2]{IzmailovSolodov2008}:
If the Newton step $d_k$ can be computed
and satisfies
\begin{equation}
	\label{eq:newton_direction_very_nice}
	\frac{ \PhiFB( z_k + d_k)}{ \PhiFB( z_k )} \le q
\end{equation}
(with a fixed parameter $q \in (0,1)$),
we perform the Newton step $z_{k+1} = z_k + d_k$.
If the Newton system is not solvable or if its solution $d_k$
violates an angle test,
we instead use $d_k := -\nabla \PhiFB(z_k)$.
Afterwards, we use an Armijo line search to obtain the step length $\alpha_k$
and update the iterate via
$z_{k+1} = z_k + \alpha_k \, d_k$.
This globalization strategy is described
in \cref{alg:global}.
\begin{algorithm}
	\KwData{parameters $q,\tau_{\text{abs}},\rho, \sigma, \beta \in(0,1)$,
	starting point $z_0\in\R^{n+\ell+m+2p}$}
	Set $k=0$\;
	\While{$\norm{F(z_k)}>\tau_{\text{abs}}$}{
		Solve $DF(z_k)d_k=-F(z_k)$ for $d_k$\;
		\eIf{$d_k$ is well defined and ratio test \eqref{eq:newton_direction_very_nice} is satisfied}{
			Set $z_{k+1}=z_k+d_k$\;
		}{
			\If{$d_k$ is not well defined or $\nabla\PhiFB(z_k)^\top d_k > -\rho\norm{d_k}\norm{\nabla\PhiFB(z_k)}$}{
				Set $d_k = -\nabla\PhiFB(z_k)$\;
			}
			Determine $z_{k+1}=z_k+\alpha_kd_k$ using an Armijo line search for $\PhiFB$, i.e.,
			$\alpha_k = \beta^{i_k}$, where $i_k \in \N_0$ is the smallest non-negative integer with
		$\PhiFB(z_k + \beta^{i_k} \, d_k) \le \PhiFB(z_k) + \sigma \, \beta^{i_k} \, \nabla\PhiFB(z_k)^\top d_k$\;
		}
		Set $k = k+1$\;
	}
	\caption{Globalization of the semismooth Newton method.}
	\label{alg:global}
\end{algorithm}

Due to \cref{lem:norm_equivalence,lem:equivalence}
and the proof of \cref{thm:newton_method},
the ratio test \eqref{eq:newton_direction_very_nice}
is satisfied
(and, consequently, the Newton steps are performed)
for all $z_k$
in the neighborhood of solutions satisfying the assumptions
of \cref{thm:uniform_invertibility}.
Consequently,
the convergence guarantees of \cref{alg:global}
follow along the lines of \cite[Section~3]{DeLucaFacchineiKanzow2000}, \cite[Thms.~3.4, 3.5]{IzmailovSolodov2008}:
	\begin{theorem}
		\label{thm:convergence}
		Let the sequence $(z_k)_{k \in \N}$ be given by \cref{alg:global}.
		\begin{enumerate}
			\item
				\label{item:nice_descent}
				If \eqref{eq:newton_direction_very_nice} is satisfied infinitely often,
				then $\PhiFB(z_k) \to 0$.
				In this case,
				any accumulation point of $(z_k)_{k\in\N}$ is a primal-dual M-stationary tuple.
			\item
				\label{item:stationary_point_of_merit_function}
				All accumulation points of $(z_k)_{k\in\N}$ are stationary points of
				$\PhiFB$.
			\item
				\label{item:convergence_of_whole_sequence}
				If an accumulation point $\bar z$ of $(z_k)_{k \in \N}$ satisfies the assumptions of
				\cref{thm:uniform_invertibility},
				then the entire sequence converges superlinearly towards $\bar z$.
				If, additionally, the second-order derivatives of the data functions
				$f$, $g$, $h$, $G$, and $H$ are locally Lipschitz
				continuous, then the convergence is quadratic.
		\end{enumerate}
	\end{theorem}
	\begin{proof}
		Statement~\ref{item:nice_descent} follows immediately since
				the sequence
				$(\PhiFB(z_k))_{k \in \N}$ is decreasing.

		Let us prove statement~\ref{item:stationary_point_of_merit_function}.
				In case where the assumption from statement~\ref{item:nice_descent} holds, this is clear.
				Otherwise, by considering the tail of the sequence,
				we may assume that \eqref{eq:newton_direction_very_nice} is never satisfied.
				Now, \cref{alg:global} reduces to a line-search method
				and the search directions satisfy the angle condition.

				For a convergent subsequence $z_{k_l} \to z^*$,
				we will show that $z^*$ is a stationary point of $\PhiFB$.
				To this end, we will distinguish three cases.

				\textit{Case 1}, $\liminf_{l \to \infty} \norm{d_{k_l}} = 0$:
				By selecting a further subsequence (without relabeling),
				we have $d_{k_l} \to 0$.
				In every step of the algorithm, we have
				\begin{equation*}
					DF(z_{k_l}) \, d_{k_l} = -F(z_{k_l})
					\qquad\text{or}\qquad
					d_{k_l} = -\nabla \PhiFB(z_{k_l}).
				\end{equation*}
				If the second equation holds infinitely often,
				we immediately get $\nabla\PhiFB(z^*) = 0$.
				Otherwise, the first equation holds infinitely often.
				The convergence $z_{k_l} \to z^*$
				implies that the Newton derivatives $DF(z_{k_l})$
				are bounded.
				This gives $F(z_{k_l}) \to 0$ (along a subsequence).
				Thus, $F(z^*) = 0$,
				see \cref{lem:phi_properties}~\ref{item:phi_cont_2},
				and therefore
				\cref{lem:equivalence} implies
				$\PhiFB(z^*) = 0$ which yields $\nabla\PhiFB(z^*) = 0$.

				\textit{Case 2}, $\liminf_{l \to \infty} \alpha_{k_l} \, \norm{d_{k_l}} > 0$:
				The sequence $(\PhiFB(z_k))_{k \in \N}$ is decreasing and bounded from below by zero,
				therefore it converges.
				The Armijo condition ensures
				$\alpha_k \, \nabla\PhiFB(z_k)^\top d_k \to 0$.
				In every step of the algorithm,
				the angle condition
				$\nabla\PhiFB(z_k)^\top d_k \le -\rho \, \norm{\nabla\PhiFB(z_k)} \, \norm{d_k}$
				is satisfied,
				thus we get
				$\alpha_k \, \norm{\nabla\PhiFB(z_k)} \, \norm{d_k} \to 0$.
				From the condition of Case~2,
				this implies $\nabla\PhiFB(z_{k_l}) \to 0$
				and, therefore, the stationarity of $z^*$.

				\textit{Case 3}, $\liminf_{l \to \infty} \norm{d_{k_l}} > 0$
				and
				$\liminf_{l \to \infty} \alpha_{k_l} \, \norm{d_{k_l}} = 0$:
				By picking a further subsequence (without relabeling),
				we get
				$\alpha_{k_l} \, \norm{d_{k_l}} \to 0$.
				Since $\norm{d_{k_l}}$
				is bounded away from zero,
				this also yields $\alpha_{k_l} \to 0$.
				In particular, for $l$ large enough,
				$\beta^{-1} \, \alpha_{k_l}$
				violates the Armijo condition,
				i.e.,
				\begin{equation*}
					\PhiFB( z_{k_l} + \beta^{-1} \, \alpha_{k_l} \, d_{k_l} )
					-
					\PhiFB( z_{k_l} )
					>
					\sigma \, \beta^{-1} \, \alpha_{k_l} \, \nabla\PhiFB(z_{k_l})^\top d_{k_l}
					.
				\end{equation*}
				Using the mean value theorem on the left-hand side,
				we get $\xi_l \in (0, \beta^{-1} \, \alpha_{k_l})$
				such that
				\begin{equation}
					\label{eq:armijo}
					\beta^{-1} \, \alpha_{k_l} \, \nabla\PhiFB(z_{k_l} + \xi_l \, d_{k_l})^\top d_{k_l}
					>
					\sigma \, \beta^{-1} \, \alpha_{k_l} \, \nabla\PhiFB(z_{k_l})^\top d_{k_l}
					.
				\end{equation}
				Note that $\xi_l \, d_{k_l} \to 0$.
				The continuous function $\nabla\PhiFB$ is uniformly continuous
				on compact neighborhoods of $z^*$.
				Thus, for every $\varepsilon > 0$,
				there is $L \in \N$ such that
				\begin{equation*}
					\forall l \ge L\colon\quad
					\norm{
						\nabla\PhiFB(z_{k_l})
						-
						\nabla\PhiFB(z_{k_l} + \xi_l \, d_{k_l})
					}
					\le \varepsilon.
				\end{equation*}
				Using this inequality in \eqref{eq:armijo},
				we get
%				 \begin{equation*}
%				 	\nabla\PhiFB(z_{k_l})^\top d_{k_l}
%				 	+
%				 	\varepsilon \, \norm{d_{k_l}}
%				 	\ge
%				 	\nabla\PhiFB(z_{k_l} + \xi_l \, d_{k_l})^\top d_{k_l}
%				 	>
%				 	\sigma \, \nabla\PhiFB(z_{k_l})^\top d_{k_l}
%				 	,
%				 \end{equation*}
%				 i.e.,
				\begin{equation*}
					(1-\sigma) \, \nabla\PhiFB(z_{k_l})^\top d_{k_l}
					+
					\varepsilon \, \norm{d_{k_l}}
					>
					0
				\end{equation*}
				for all $l$ large enough.
				Exploiting the angle condition once more yields
				% \begin{equation*}
				% 	(\varepsilon - (1-\sigma) \, \rho \, \norm{\nabla\PhiFB(z_{k_l})}) \, \norm{d_{k_l}}
				% 	>
				% 	0,
				% \end{equation*}
				% i.e.,
				\begin{equation*}
					\varepsilon  > (1-\sigma) \, \rho \, \norm{\nabla\PhiFB(z_{k_l})}
				\end{equation*}
				for all $l$ large enough.
				Since $\varepsilon > 0$ was arbitrary,
				this gives
				$\norm{\nabla\PhiFB(z_{k_l})} \to 0$
				and, therefore, $z^*$ is a stationary point of $\PhiFB$.

			Finally, we show statement~\ref{item:convergence_of_whole_sequence}.
				From the proof of \cref{thm:newton_method} in conjunction with 
				\cref{lem:norm_equivalence,lem:equivalence},
				we obtain an $\varepsilon > 0$ such that
				\eqref{eq:newton_direction_very_nice}
				as well as $\norm{(z_k + d_k) - \bar z} \le \norm{z_k - \bar z}$
				are satisfied for all $z_k \in B_\varepsilon(\bar z)$.
				This means that the semismooth Newton step is accepted
				and the next iterate $z_{k+1} = z_k + d_k$ also belongs to $B_\varepsilon(\bar z)$.
				Consequently, \cref{alg:global} becomes a semismooth Newton method
				and the assertion follows from \cref{thm:newton_method}.
	\end{proof}
	Clearly, every primal-dual M-stationary tuple $\bar z$
	is a stationary point of $\PhiFB$ since $\PhiFB(\bar z) = 0$.
	However, giving assumptions under which the converse implication is true
	seems to be challenging,
	see also \cref{subsec:non-M}.
	For the solution of NCPs, this question has been adressed in \cite[Section~4]{DeLucaFacchineiKanzow1996}.

\section{Convergence for linear-quadratic problems beyond MPCC-LICQ}
\label{sec:linear_quadratic}

We consider the linear-quadratic case,
i.e., we assume that the function
$f$ is quadratic and that the mappings
$g$, $h$, $G$, and $H$ are affine.
Due to the complementarity constraints,
the solution of \eqref{eq:MPCC}
is still very challenging.
On the other hand, it follows from
\cite[Theorem~3.5, Proposition~3.8]{FlegelKanzow2005b}
that local minimizers of the associated problem
\eqref{eq:MPCC} are M-stationary without further
assumptions. This makes the search for M-stationary
points even more attractive.
Our goal is to verify that
one-step convergence of a modification of our Newton method
is possible under
a weaker constraint qualification than MPCC-LICQ.

Let an M-stationary point $\bar x\in\R^n$ of \eqref{eq:MPCC}
with multiplier 
$(\bar\lambda,\bar\rho,\bar\mu,\bar\nu)\in\Lambda^\textup{M}(\bar x)$
be given and set 
$\bar z = (\bar x, \bar\lambda, \bar\eta, \bar\mu, \bar\nu)$.
We require that
the matrix
\begin{equation}
	\label{eq:matrix_relaxed_MPCC_LICQ}
	\begin{bmatrix}
		g'(\bar x)_{I^g_+(\bar x,\bar\lambda)}
		\\
		h'(\bar x)
		\\
		G'(\bar x)_{I^{0+}(\bar x)\cup I^{00}_{\pm\R}(\bar x, \bar\mu, \bar\nu)}
		\\
		H'(\bar x)_{I^{+0}(\bar x)\cup I^{00}_{\R\pm}(\bar x, \bar\mu, \bar\nu)}
	\end{bmatrix}
\end{equation}
possesses full row rank,
where we used the multiplier-dependent index sets from \eqref{eq:more_index_sets}
and
\begin{equation*}
	I^g_+(\bar x,\bar\lambda)
	:=
	\set{
		i \in I^g(\bar x) \given \bar\lambda_i > 0
	}
	.
\end{equation*}
Clearly, this condition is, in general, weaker than MPCC-LICQ.
Further, we assume that
MPCC-SSOC holds
at $\bar x$
w.r.t.\ $(\bar\lambda, \bar\eta, \bar\mu, \bar\nu)$.

Let
$( x, \lambda, \eta, \mu, \nu)$
denote the current iterate.
We will assume that it is close to the solution
$(\bar x, \bar\lambda, \bar\eta, \bar\mu, \bar\nu)$.
Then arguing as in the proof of \cref{thm:uniform_invertibility},
using the active-set interpretation from the end of
\cref{sec:newton_method}
as well as the linear-quadratic structure of the problem,
the next iterate
$( x^+, \lambda^+, \eta^+, \mu^+, \nu^+)$
is given by solving the linear system
\begin{subequations}
	\label{eq:some_large_system}
	\begin{align}
		\nabla_x \LL(x^+, \lambda^+, \eta^+, \mu^+, \nu^+ )
		&=
		0,
		\\
		g(x^+)_{I^\ell_1} &= 0,
		&
		(\lambda^+)_{I^\ell \setminus I^\ell_1} &= 0,
		\\
		h(x^+) &= 0,
		\\
		G(x^+)_{I^p_\mu} &= 0,
		&
		(\mu^+)_{I^p \setminus I^p_\mu} &= 0,
		\\
		H(x^+)_{I^p_\nu} &= 0,
		&
		(\nu^+)_{I^p \setminus I^p_\nu} &= 0
		.
	\end{align}
\end{subequations}
Here, the index sets
$I^\ell_1$, $I^p_\mu$, and $I^p_\nu$
are constructed similarly as in the proof of \cref{thm:uniform_invertibility}
and satisfy \eqref{eq:properties_index_sets}.

In the following, we argue that the index sets
$I^\ell_1$, $I^p_\mu$, and $I^p_\nu$
can be modified such that
$\bar z = (\bar x, \bar\lambda, \bar\eta, \bar\mu, \bar\nu)$
is the unique solution of \eqref{eq:some_large_system}.
For arbitrary index sets
$I^\ell_1$, $I^p_\mu$, and $I^p_\nu$,
we have the implications
\begin{subequations}
	\label{eq:implication_linear_quadratic}
	\begin{align}
		\label{eq:implication_linear_quadratic_1}
		\text{%
			the sets 
			$I_1^\ell$, $I^p_\mu$, $I^p_\nu$
			satisfy \eqref{eq:properties_index_sets}%
		}
		&\;\Rightarrow\;
		\bar z \text{ solves \eqref{eq:some_large_system},}
		\\
		\label{eq:implication_linear_quadratic_2}
		\myrbraces*{
			\begin{aligned}
				I^g_+(\bar x, \bar \lambda) &= I^\ell_1 \\
				I^{0+}(\bar x) \cup I^{00}_{\pm\R}(\bar x, \bar \mu, \bar \nu) &= I^p_\mu \\
				I^{+0}(\bar x) \cup I^{00}_{\R\pm}(\bar x, \bar \mu, \bar \nu) &= I^p_\nu
			\end{aligned}
			\;
		}
		&\;\Rightarrow\;
		\text{\eqref{eq:some_large_system} is uniquely solvable}
		.
	\end{align}
\end{subequations}
Note that \eqref{eq:implication_linear_quadratic_2} follows
from the assumption that \eqref{eq:matrix_relaxed_MPCC_LICQ} has full row rank.

In particular,
system \eqref{eq:some_large_system} is, in general, not uniquely solvable.
If the system has multiple solutions,
\cref{lem:saddle_point_system} and MPCC-SSOC
imply
that the matrix
\begin{equation}
	\label{eq:ugly_matrix}
	\begin{bmatrix}
		g'(\bar x)_{I^\ell_1}
		\\
		h'(\bar x)
		\\
		G'(\bar x)_{I^p_\mu}
		\\
		H'(\bar x)_{I^p_\nu}
	\end{bmatrix}
\end{equation}
does not possess
full row rank.
Note that this matrix might possess more rows than \eqref{eq:matrix_relaxed_MPCC_LICQ}.

We will see that it is possible to
remove one index from one of the index sets
$I_1^\ell$, $I^p_\mu$, $I^p_\nu$
such that the inclusions \eqref{eq:properties_index_sets}
are still satisfied,
as long as \eqref{eq:some_large_system} is not uniquely solvable.
Thus, \eqref{eq:implication_linear_quadratic}
will imply that we can find the solution $\bar z$
using this strategy repeatedly.

We can filter out the linearly dependent rows
from the knowledge of our current iterate.
The possible linearly dependent rows in \eqref{eq:ugly_matrix}
correspond to the index sets
\begin{equation}
	\label{eq:bad_guys}
	I^\ell_1 \setminus I^g_+(\bar x, \bar\lambda)
	,\qquad
	I^p_\mu \setminus \paren[\big]{I^{0+}(\bar x)\cup I^{00}_{\pm\R}(\bar x, \bar\mu, \bar\nu)}
	,\qquad
	I^p_\nu \setminus \paren[\big]{I^{+0}(\bar x)\cup I^{00}_{\R\pm}(\bar x, \bar\mu, \bar\nu)}
	.
\end{equation}
Using \eqref{eq:properties_index_sets},
these indices are contained in
\begin{align*}
	I^g_0(\bar x, \bar\lambda)
	&:=
	I^g(\bar x)\setminus I^g_+(\bar x,\bar\lambda),
	\\
	I^{00}_{0\R}(\bar x, \bar\mu, \bar\nu) 
	&:= 
	I^{00}(\bar x) \setminus I^{00}_{\pm\R}(\bar x, \bar\mu, \bar\nu),
	\\
	I^{00}_{\R0}(\bar x, \bar\mu, \bar\nu) 
	&:= 
	I^{00}(\bar x) \setminus I^{00}_{\R\pm}(\bar x, \bar\mu, \bar\nu),
\end{align*}
respectively.
Now, we use the following procedure:
First, we sort the indices
\[
	\begin{aligned}
	&i \in I^\ell_1 &	&\text{ according to } \lambda_i \\
	&j \in I^p_\mu &	&\text{ according to } \max( \abs{\mu_j}, \abs{H_j(x)} ) \\
	&j \in I^p_\nu &	&\text{ according to } \max( \abs{\nu_j}, \abs{G_j(x)} )
	\end{aligned}
\]
in increasing order.
If the current iterate is sufficiently close to the solution,
the above list will contain the problematic indices from \eqref{eq:bad_guys} at the top.
Then we can remove the indices one-by-one
from the corresponding index sets,
until the system \eqref{eq:some_large_system} becomes
uniquely solvable.
Note that this modification of the index sets
ensures that \eqref{eq:properties_index_sets} stays valid.
Hence, the solution $\bar z$
remains to be a solution of \eqref{eq:some_large_system}, cf.\ \eqref{eq:implication_linear_quadratic_1}.
Since the matrix \eqref{eq:matrix_relaxed_MPCC_LICQ} has full row rank,
this process stops if all indices from \eqref{eq:bad_guys} are removed
(or earlier).

\begin{example}
	\label{ex:scheel_scholtes}
	We consider the classical example \cite[Example~3]{ScheelScholtes2000}
	with a quadratic regularization term.
	That is, we have
	$n = 3$, $\ell = 2$, $m = 0$, as well as $p = 1$,
	and the functions are given by
	\begin{equation*}
		f(x) = x_1 + x_2 - x_3 + \frac c2 \, \norm{x}^2,
		\;
		g(x) = \begin{pmatrix}-4 \, x_1 + x_3 \\ -4 \, x_2 + x_3\end{pmatrix},
		\;
		G(x) = x_1,
		\;
		H(x) = x_2,
	\end{equation*}
	where $c > 0$ is the regularization parameter.
	The global minimizer is $\bar x = 0$,
	and this point is M-stationary with multipliers
	$\bar \lambda = (3/4, 1/4)$,
	$\bar \mu = 2$, $\bar \nu = 0$.
	An alternative set of multipliers is
	$\tilde \lambda = (1/4, 3/4)$,
	$\tilde \mu = 0$, $\tilde \nu = 2$.
	Since all four constraints are active in $\bar x$,
	MPCC-LICQ cannot be satisfied.
	However, the matrix \eqref{eq:matrix_relaxed_MPCC_LICQ} is given by
	\begin{equation*}
		\begin{bmatrix}
			g'(\bar x)_{\{1,2\}} \\
			G'(\bar x)
		\end{bmatrix}
		=
		\begin{pmatrix}
			-4 & 0 & 1 \\
			0 & -4 & 1 \\
			1 & 0 & 0
		\end{pmatrix}
	\end{equation*}
	and this matrix possesses full row rank.
	Since $\nabla^2f(\bar x)$ is positive definite,
	MPCC-SSOC holds as well.
	Hence, the above theory applies
	and we obtain one-step convergence
	if the initial guess is sufficiently close
	to $(\bar x, \bar\lambda, \bar\mu, \bar \nu)$,
	see \cref{rem:one_step_convergence_for_linear_quadratic_problems} as well.
\end{example}

The application of this idea to problems which are not linear-quadratic
is subject to future research. 
We expect that the above idea can be generalized easily to problems
with affine constraints.
In this situation,
every local minimizer is M-stationary.
Hence, it is suitable to solve such problems
with an algorithm capable to find M-stationary points.

\section{Numerical results}
\label{sec:numerics}

\subsection{Implementation and parameters}
\label{sec:implementation}

Numerical experiments were carried out using MATLAB (R2020a).
Therefore, we implemented \cref{alg:global}
with the modifications discussed in \cref{sec:linear_quadratic}.
These modifications allow us to solve the system $DF(z_k)d_k = -F(z_k)$
if $DF(z_k)$ is not invertible but $z_k$ is near the solution
in the case of linear-quadratic MPCCs.
For the parameters in \cref{alg:global}, we chose 
$q=0.999$, $\tau_\textup{abs}=10^{-11}$, $\rho=10^{-3}$,
and $\sigma = \beta = 1/2$.

	We also implemented \cite[Algorithm~4.1]{GuoLinYe2015}.
	Since we are interested in M-stationary points,
	we chose $F$ and $W$ according to \cite[(16), (17)]{GuoLinYe2015}.
	We used the parameters $\sigma=1$ and $\eta=1/10$,
	which are the same as in
	\cite[Section~6]{GuoLinYe2015},
	but changed the termination condition from
	$\min(\norm{F(w^k)},\norm{d^k})\leq10^{-6}$
	to 
	$\min(\norm{F(w^k)},\norm{d^k})\leq10^{-11}$
	in order to have greater similarity with
	the termination condition in \cref{alg:global}.
	The quadratic program in \cite[Algorithm~4.1]{GuoLinYe2015} 
	was solved using \texttt{quadprog} in MATLAB,
	with tolerances
	\texttt{ConstraintTolerance}, \texttt{OptimalityTolerance} and \texttt{StepTolerance}
	set to $10^{-13}$, in order to achieve greater accuracy.

	For both algorithms, we used random starting points
	$z_0=(x_0,\lambda_0,\eta_0,\mu_0,\nu_0)\in\R^{n+\ell+m+2p}$,
	which were chosen from a uniform distribution on $[-n,n]^{n+\ell+m+2p}$,
	with the exception of choosing $\lambda_0\in[0,n]^\ell$
	for \cite[Algorithm~4.1]{GuoLinYe2015} 
	in order to guarantee that the starting point is feasible.
	Similarly, we used a uniform distribution on $[0,n]$
	for the initialization of the slack variables
	in \cite[Algorithm~4.1]{GuoLinYe2015}.
	Unless stated otherwise, we executed $1000$ runs of each
	algorithm for a problem.
	In  each of the runs, both algorithms were initialized with
	the same random starting point,
	which was chosen independently across the runs.
	Each run was aborted if it had not finished within $1000$ iterations.

\subsection{A toy example}

As a first small toy example we consider the
linear-quadratic MPCC from \cref{ex:scheel_scholtes}
with the choice $c=1/10$ for the regularization parameter.
Recall that this MPCC has a local minimizer at
$\bar x=(0,0,0)$, which is an M- but not an
S-stationary point.
As discussed previously, the matrix \eqref{eq:matrix_relaxed_MPCC_LICQ}
possesses full row rank and MPCC-SSOC holds, too.
Thus, our assumptions for \cref{alg:global}
and its modifications discussed in \cref{sec:linear_quadratic}
are satisfied.

	We performed numerical tests for this MPCC as described in
	\cref{sec:implementation}.
	For \cref{alg:global}, the termination criterion was always reached
	and the average euclidean distance of the calculated solution
	to $\bar x$ was $5.6\cdot 10^{-17}$.
	We could also observe the local one-step convergence that we expected from
	\cref{rem:one_step_convergence_for_linear_quadratic_problems}.
	The average number of iterations was $7.19$,
	and each run took $0.007$ seconds in average.

	We also performed numerical tests for this toy problem with
	\cite[Algorithm~4.1]{GuoLinYe2015}.
	Here, the termination criterion was always reached
	and the average euclidean distance of the calculated solution 
	to $\bar x$ was $2.2\cdot 10^{-7}$.
	In all runs the termination criterion was met because the norm 
	of the search direction $d_k$ became sufficiently small
	(and not because $\norm{F(w_k)}$ was sufficiently small).
	The average number of iterations was $15.04$,
	and each run took $0.026$ seconds in average.

\subsection{Convergence towards non-M-stationary points}
\label{subsec:non-M}
	For each parameter $\varepsilon \ge 0$, we consider the two-dimensional problem
	\begin{align*}
		\min_x \quad&\tfrac12 \, \norm{x - b_\varepsilon}^2 \\
		\text{s.t.}\quad& x_1 \ge 0, \; x_2 \ge 0, \; x_1 \, x_2 = 0,
	\end{align*}
	where $b_\varepsilon = (1, -\varepsilon)$.
	We denote the corresponding merit function from \cref{sec:glob}
	by $\PhiFB^\varepsilon$.
	The unique global minimizer of the MPCC is $\bar x = (1,0)$
	(independent of $\varepsilon \ge 0$)
	and this point is strongly stationary.
	In case $\varepsilon > 0$, there is no further M-stationary point,
	but for $\varepsilon = 0$, the point $\tilde x = (0,0)$
	is M-stationary.
	By attaching the unique multipliers,
	we obtain that the point
	$\tilde z = (0,0,1,0)$
	is an isolated local minimizer of $\PhiFB^0$
	and $\PhiFB^0(\tilde z) = 0$ since $\tilde z$ is an M-stationary primal-dual tuple.
	Hence, it can be expected that $\PhiFB^\varepsilon$ still has a local minimizer
	in the neighborhood of $\tilde z$ for a small perturbation $\varepsilon > 0$,
	but this local minimizer cannot be M-stationary.
	The same argument can be applied to the function $\theta$ from \cite[(21)]{GuoLinYe2015}.

	For the numerical computations, we chose $\varepsilon = 0.2$.
	We executed 1000 runs for both algorithms.
	In $659$ cases, \cref{alg:global} converged towards a non-M-stationary primal-dual tuple,
	whereas this happened in $250$ runs of \cite[Algorithm~4.1]{GuoLinYe2015}.
	In all the other cases, the global solution $\bar x$ has been found.
	One should bear in mind that these numbers strongly depend
	on the random distribution of the initial points,
	see \cref{sec:implementation} for our choice.

	We expect that this instability of M-stationary points w.r.t.\ perturbations
	will also impede the convergence of similar algorithms
	which try to compute M-stationary points.

\subsection{Optimal control of a discretized obstacle problem}

We consider a discretized version of the optimization problem from \cite[Section~6.1]{Wachsmuth2013:2}.
This is an infinite dimensional MPCC for which strong stationarity does not hold
at the uniquely determined minimizer.
We will see that the same property holds for its discretization.

Let us fix a discretization parameter $N \in \N$.
The optimization variable $x \in \R^{3N}$
is partitioned as $x = (y, u, \xi)$.
The discretized problem uses the data functions
\begin{align*}
	f(x) &= \tfrac12 \, \norm{y}^2 + e^\top y + \tfrac12 \, \norm{u}^2,
	&
	g(x) &= -u,
	&
	h(x) &= A \, y - u + \xi,
	\\
	&&
	G(x) &= -y,
	&
	H(x) &= \xi,
\end{align*}
where $e := (1,\ldots,1) \in \R^N$ is the all-ones vector and the matrix $A$ is given entrywise:
\begin{equation*}
	\forall i,j\in\{1,\ldots,N\}\colon\quad
	A_{i,j} :=
	\begin{cases}
		2 & \text{if } i = j, \\
		-1 & \text{if } i = j \pm 1, \\
		0 & \text{else}.
	\end{cases}
\end{equation*}
Up to the scaling, this matrix arises from the finite-difference discretization
of the one-dimensional Laplacian.

One can check that $\bar x = 0$ is the unique global minimizer.
Since all constraints are affine, $\bar x$ is an M-stationary point of this program.
Furthermore, there are no weakly stationary points (i.e., feasible points satisfying \eqref{eq:M_St_x}--\eqref{eq:M_St_nu})
different from $\bar x$.
The M-stationary multipliers $(\lambda,\eta,\mu,\nu)$ associated with $\bar x$ are not unique.
Indeed, for every diagonal matrix $D \in \R^{N \times N}$
with diagonal entries from $\{0,1\}$,
one can check that the solution of the system
\begin{equation*}
	\begin{bmatrix}
		A & I \\ I-D & D
	\end{bmatrix}
	\begin{bmatrix}
		\nu \\ \mu
	\end{bmatrix}
	=
	\begin{bmatrix}
		e \\ 0
	\end{bmatrix}
	,
	\;
	\nu = -\eta = \lambda
\end{equation*}
yields M-stationary multipliers,
and all multipliers are obtained by this construction.
Let us mention that none of these multipliers solves the associated system of S-stationarity
since for each $i\in\N$, either $\mu_i$ or $\nu_i$ is positive.
Thus, $\bar x$ is not an S-stationary point.
For all these multipliers,
the matrix \eqref{eq:matrix_relaxed_MPCC_LICQ} possesses full row rank.
On the other hand, MPCC-LICQ is clearly violated at $\bar x$ since this point
is a local minimizer of the given MPCC which is not S-stationary.
Finally, one can check that MPCC-SSOC is valid at $\bar x$ w.r.t.\ all the
multipliers characterized above.

	We performed numerical tests as described in \cref{sec:implementation}
	with the discretization parameter $N=256$.
	For \cref{alg:global}, the solution has been found in every
	run, with an average error of $6.7\cdot 10^{-31}$ (measured in the euclidean distance).
	The average number of iterations was $13.38$,
	and each run took $2.36$ seconds in average.
	Again, we could also observe the local one-step convergence
	that we expected from 
	\cref{rem:one_step_convergence_for_linear_quadratic_problems}.
	The associated multipliers differed from run to run.
	That behavior, however, had to be
	expected since the associated multiplier set $\Lambda^\textup{M}(\bar x)$
	is not a singleton.

	In comparison, \cite[Algorithm~4.1]{GuoLinYe2015}
	did not converge within $1000$ iterations in each run
	(we only executed $10$ runs of the algorithm instead of the usual
	$1000$ runs due to the long runtime).
	Each run took $7298$ seconds in average,
	and the average of the error $\norm{x_{1000}-\bar x}$
	of the final iterate $x_{1000}$ was $4538$.
	One reason for the long runtime of
	\cite[Algorithm~4.1]{GuoLinYe2015}
	is that a quadratic program in $14N=3584$ dimensions needs to be 
	solved in each iteration.

	We also conducted numerical tests for a very coarse discretization
	with $N=4$.
	Here, \cite[Algorithm~4.1]{GuoLinYe2015}
	performed significantly better.
	In all runs, the termination criterion was met
	because the search direction $d_k$ became sufficiently small
	(and not because $\norm{F(w_k)}$ was sufficiently small).
	The average euclidean distance of the calculated solution
	to $\bar x$ was $7.8\cdot 10^{-4}$.
	The average number of iterations was $41.75$,
	and each run took $0.149$ seconds in average.

	For this coarse discretization (with $N=4$),
	\cref{alg:global} also performed reasonably well.
	In all runs, the termination criterion was met
	and the average euclidean distance of the calculated solution
	to $\bar x$ was $6.9\cdot 10^{-16}$.
	The average number of iterations was $2.91$,
	and each run took $0.003$ seconds in average.

\section{Conclusion and outlook}\label{sec:conclusions}

We demonstrated that the M-stationarity system of a mathematical
problem with complementarity constraints can be reformulated as a system of nonsmooth
and discontinuous equations, see \cref{sec:M_St_as_system}. 
It has been shown that this system can be solved
with the aid of a semismooth Newton method. Local fast convergence to M-stationary
points can be guaranteed under validity of MPCC-SSOC and MPCC-LICQ, 
see \cref{thm:local_convergence}, where the
latter assumption can be weakened in case of linear-quadratic problems, 
see \cref{sec:linear_quadratic}.
Furthermore, we provided a reasonable globalization strategy, see \cref{sec:glob}.
There is some hope that similar to \cref{sec:linear_quadratic}, it
is possible to weaken MPCC-LICQ in the setting of nonlinear constraints if
only validity of a suitable constant rank assumption can be guaranteed.
Clearly, the fundamental ideas of this paper are not limited to mathematical
programs with complementarity constraints but can be adjusted in order to
suit other problem classes from disjunctive programming such as 
vanishing-constrained, switching-constrained, or cardinality-constrained
programs, see \cite{Mehlitz2019b} for an overview. 
It remains an open question to what extent the theory of this
paper can be generalized to infinite-dimensional complementarity-constrained
optimization problems.

%%fakesection: Bibliography
\bibliographystyle{siamplain}
\bibliography{world}

\begin{thebibliography}{10}

\bibitem{BenkoGfrerer2016}
{\sc M.~Benko and H.~Gfrerer}, {\em An {S}{Q}{P} method for mathematical
  programs with complementarity constraints with strong convergence
  properties}, Kybernetika, 52 (2016), pp.~169--208,
  \url{https://doi.org/10.14736/kyb-2016-2-0169}.

\bibitem{ChenNashedQi2000}
{\sc X.~Chen, Z.~Nashed, and L.~Qi}, {\em Smoothing methods and semismooth
  methods for nondifferentiable operator equations}, {SIAM} Journal on
  Numerical Analysis, 38 (2000), pp.~1200--1216,
  \url{https://doi.org/10.1137/s0036142999356719}.

\bibitem{Clason2017}
{\sc C.~Clason}, {\em Nonsmooth analysis and optimization}, 2018,
  \url{https://arxiv.org/abs/1708.04180v2}.

\bibitem{DeLucaFacchineiKanzow1996}
{\sc T.~De~Luca, F.~Facchinei, and C.~Kanzow}, {\em A semismooth equation
  approach to the solution of nonlinear complementarity problems}, Mathematical
  Programming, 75 (1996), pp.~407--439,
  \url{https://doi.org/10.1007/BF02592192}.

\bibitem{DeLucaFacchineiKanzow2000}
{\sc T.~De~Luca, F.~Facchinei, and C.~Kanzow}, {\em A {T}heoretical and
  {N}umerical {C}omparison of {S}ome {S}emismooth {A}lgorithms for
  {C}omplementarity {P}roblems}, Computational Optimization and Applications,
  16 (2000), pp.~173--205, \url{https://doi.org/10.1023/A:1008705425484}.

\bibitem{FacchineiFischerKanzowPeng1999}
{\sc F.~Facchinei, A.~Fischer, C.~Kanzow, and J.-M. Peng}, {\em A simply
  constrained optimization reformulation of {K}{K}{T} systems arising from
  variational inequalities}, Applied Mathematics and Optimization, 40 (1999),
  pp.~19--37, \url{https://doi.org/10.1007/s002459900114}.

\bibitem{FacchineiSoares1997}
{\sc F.~Facchinei and J.~Soares}, {\em A {N}ew {M}erit {F}unction {F}or
  {N}onlinear {C}omplementarity {P}roblems {A}nd {A} {R}elated {A}lgorithm},
  SIAM Journal on Optimization, 7 (1997), pp.~225--247,
  \url{https://doi.org/10.1137/S1052623494279110}.

\bibitem{FlegelKanzow2005b}
{\sc M.~L. Flegel and C.~Kanzow}, {\em On {M}-stationary points for
  mathematical programs with equilibrium constraints}, Journal of Mathematical
  Analysis and Applications, 310 (2005), pp.~286 -- 302,
  \url{https://doi.org/10.1016/j.jmaa.2005.02.011}.

\bibitem{FlegelKanzow2006:1}
{\sc M.~L. Flegel and C.~Kanzow}, {\em A direct proof for {M}-stationarity
  under {MPEC}-{GCQ} for mathematical programs with equilibrium constraints},
  in Optimization with multivalued mappings, vol.~2 of Springer Optim. Appl.,
  Springer, New York, 2006, pp.~111--122,
  \url{https://doi.org/10.1007/0-387-34221-4_6}.

\bibitem{FletcherLeyfferRalphScholtes2006}
{\sc R.~Fletcher, S.~Leyffer, D.~Ralph, and S.~Scholtes}, {\em Local
  {C}onvergence of {S}{Q}{P} {M}ethods for {M}athematical {P}rograms with
  {E}quilibrium {C}onstraints}, SIAM Journal on Optimization, 17 (2006),
  pp.~259--286, \url{https://doi.org/10.1137/S1052623402407382}.

\bibitem{FukushimaTseng2002}
{\sc M.~Fukushima and P.~Tseng}, {\em An {I}mplementable {A}ctive-{S}et
  {A}lgorithm for {C}omputing a {B}-{S}tationary {P}oint of a {M}athematical
  {P}rogram with {L}inear {C}omplementarity {C}onstraints}, SIAM Journal on
  Optimization, 12 (2002), pp.~724--739,
  \url{https://doi.org/10.1137/S1052623499363232}.

\bibitem{Galantai2012}
{\sc A.~Gal{\'a}ntai}, {\em Properties and construction of {N}{C}{P}
  functions}, Computational Optimization and Applications, 52 (2012),
  pp.~805--824, \url{https://doi.org/10.1007/s10589-011-9428-9}.

\bibitem{Gfrerer2014}
{\sc H.~Gfrerer}, {\em Optimality {C}onditions for {D}isjunctive {P}rograms
  {B}ased on {G}eneralized {D}ifferentiation with {A}pplication to
  {M}athematical {P}rograms with {E}quilibrium {C}onstraints}, SIAM Journal on
  Optimization, 24 (2014), pp.~898--931,
  \url{https://doi.org/10.1137/130914449}.

\bibitem{GuoLinYe2013}
{\sc L.~Guo, G.-H. Lin, and J.~J. Ye}, {\em Second-{O}rder {O}ptimality
  {C}onditions for {M}athematical {P}rograms with {E}quilibrium {C}onstraints},
  Journal of Optimization Theory and Applications, 158 (2013), pp.~33--64,
  \url{https://doi.org/10.1007/s10957-012-0228-x}.

\bibitem{GuoLinYe2015}
{\sc L.~Guo, G.-H. Lin, and J.~J. Ye}, {\em Solving {M}athematical {P}rograms
  with {E}quilibrium {C}onstraints}, Journal of Optimization Theory and
  Applications, 166 (2015), pp.~234--256,
  \url{https://doi.org/10.1007/s10957-014-0699-z}.

\bibitem{HintermuellerItoKunisch2002}
{\sc M.~Hintermüller, K.~Ito, and K.~Kunisch}, {\em The primal-dual active set
  strategy as a semismooth newton method}, {SIAM} Journal on Optimization, 13
  (2002), pp.~865--888, \url{https://doi.org/10.1137/s1052623401383558}.

\bibitem{HoheiselKanzowSchwartz2013}
{\sc T.~Hoheisel, C.~Kanzow, and A.~Schwartz}, {\em Theoretical and numerical
  comparison of relaxation methods for mathematical programs with
  complementarity constraints}, Mathematical Programming, 137 (2013),
  pp.~257--288, \url{https://doi.org/10.1007/s10107-011-0488-5}.

\bibitem{HuRalph2004}
{\sc X.~M. Hu and D.~Ralph}, {\em Convergence of a {P}enalty {M}ethod for
  {M}athematical {P}rogramming with {C}omplementarity {C}onstraints}, Journal
  of Optimization Theory and Applications, 123 (2004), pp.~365--390,
  \url{https://doi.org/10.1007/s10957-004-5154-0}.

\bibitem{HuangYangZhu2006}
{\sc X.~X. Huang, X.~Q. Yang, and D.~L. Zhu}, {\em A {S}equential {S}mooth
  {P}enalization {A}pproach to {M}athematical {P}rograms with {C}omplementarity
  {C}onstraints}, Numerical Functional Analysis and Optimization, 27 (2006),
  pp.~71--98, \url{https://doi.org/10.1080/01630560500538797}.

\bibitem{ItoKunisch2008}
{\sc K.~Ito and K.~Kunisch}, {\em Lagrange Multiplier Approach to Variational
  Problems and Applications}, Society for Industrial and Applied Mathematics,
  2008, \url{https://doi.org/10.1137/1.9780898718614}.

\bibitem{IzmailovPogosyanSolodov2012}
{\sc A.~F. Izmailov, A.~L. Pogosyan, and M.~V. Solodov}, {\em Semismooth
  {N}ewton method for the lifted reformulation of mathematical programs with
  complementarity constraints}, Computational Optimization and Applications, 51
  (2012), pp.~199--221, \url{https://doi.org/10.1007/s10589-010-9341-7}.

\bibitem{IzmailovSolodov2008}
{\sc A.~F. Izmailov and M.~V. Solodov}, {\em An active-set newton method for
  mathematical programs with complementarity constraints}, {SIAM} Journal on
  Optimization, 19 (2008), pp.~1003--1027,
  \url{https://doi.org/10.1137/070690882}.

\bibitem{JudiceSheraliRibeiroFaustino2007}
{\sc J.~J. J{\'u}dice, H.~D. Sherali, I.~M. Ribeiro, and A.~M. Faustino}, {\em
  Complementarity {A}ctive-{S}et {A}lgorithm for {M}athematical {P}rogramming
  {P}roblems with {E}quilibrium {C}onstraints}, Journal of Optimization Theory
  and Applications, 134 (2007), pp.~467--481,
  \url{https://doi.org/10.1007/s10957-007-9231-z}.

\bibitem{KanzowYamashitaFukushima1997}
{\sc C.~Kanzow, N.~Yamashita, and M.~Fukushima}, {\em New {N}{C}{P}-{F}unctions
  and {T}heir {P}roperties}, Journal of Optimization Theory and Applications,
  94 (1997), pp.~115--135, \url{https://doi.org/10.1023/A:1022659603268}.

\bibitem{Kato1995}
{\sc T.~Kato}, {\em Perturbation Theory for Linear Operators}, Grundlehren der
  mathematischen Wissenschaften, Springer Berlin Heidelberg, 1995,
  \url{https://doi.org/10.1007/978-3-642-66282-9}.

\bibitem{LeyfferLopezNocedal2006}
{\sc S.~Leyffer, G.~L{\'o}pez-Calva, and J.~Nocedal}, {\em Interior {M}ethods
  for {M}athematical {P}rograms with {C}omplementarity {C}onstraints}, SIAM
  Journal on Optimization, 17 (2006), pp.~52--77,
  \url{https://doi.org/10.1137/040621065}.

\bibitem{LuoPangRalph1996}
{\sc Z.-Q. Luo, J.-S. Pang, and D.~Ralph}, {\em Mathematical Programs with
  Equilibrium Constraints}, Cambridge University Press, Cambridge, 1996.

\bibitem{LuoPangRalph1998}
{\sc Z.-Q. Luo, J.-S. Pang, and D.~Ralph}, {\em Piecewise {S}equential
  {Q}uadratic {P}rogramming for {M}athematical {P}rograms with {N}onlinear
  {C}omplementarity {C}onstraints}, in Multilevel {O}ptimization: {A}lgorithms
  and {A}pplications, A.~Migdalas, P.~M. Pardalos, and P.~V{\"a}rbrand, eds.,
  Springer, Boston, 1998, pp.~209--229,
  \url{https://doi.org/10.1007/978-1-4613-0307-7_9}.

\bibitem{Mehlitz2019b}
{\sc P.~Mehlitz}, {\em On the linear independence constraint qualification in
  disjunctive programming}, Optimization, 69 (2020), pp.~2241--2277,
  \url{https://doi.org/10.1080/02331934.2019.1679811}.

\bibitem{NocedalWright2006}
{\sc J.~Nocedal and S.~Wright}, {\em Numerical Optimization}, Springer, New
  York, 2~ed., 2006, \url{https://doi.org/10.1007/978-0-387-40065-5}.

\bibitem{OutrataKocvaraZowe1998}
{\sc J.~V. Outrata, M.~Ko{\v{c}}vara, and J.~Zowe}, {\em Nonsmooth Approach to
  Optimization Problems with Equilibrium Constraints}, Kluwer Academic,
  Dordrecht, 1998.

\bibitem{RalphWright2004}
{\sc D.~Ralph and S.~J. Wright}, {\em Some properties of regularization and
  penalization schemes for {M}{P}{E}{C}s}, Optimization Methods and Software,
  19 (2004), pp.~527--556, \url{https://doi.org/10.1080/10556780410001709439}.

\bibitem{Robinson1980}
{\sc S.~M. Robinson}, {\em Strongly {R}egular {G}eneralized {E}quations},
  Mathematics of Operations Research, 5 (1980), pp.~43--62,
  \url{https://doi.org/10.1287/moor.5.1.43}.

\bibitem{ScheelScholtes2000}
{\sc S.~Scheel and S.~Scholtes}, {\em Mathematical programs with
  complementarity constraints: Stationarity, optimality, and sensitivity},
  Mathematics of Operations Research, 25 (2000), pp.~1--22,
  \url{https://doi.org/10.1287/moor.25.1.1.15213}.

\bibitem{SunQi1999}
{\sc D.~Sun and L.~Qi}, {\em On {N}{C}{P}-functions}, Computational
  Optimization and Applications, 13 (1999), pp.~201--220,
  \url{https://doi.org/10.1023/A:1008669226453}.

\bibitem{Tseng1996}
{\sc P.~Tseng}, {\em Growth behavior of a class of merit functions for the
  nonlinear complementarity problem}, Journal of Optimization Theory and
  Applications, 89 (1996), pp.~17--37,
  \url{https://doi.org/10.1007/bf02192639}.

\bibitem{Ulbrich2002}
{\sc M.~Ulbrich}, {\em Semismooth newton methods for operator equations in
  function spaces}, {SIAM} Journal on Optimization, 13 (2002), pp.~805--841,
  \url{https://doi.org/10.1137/s1052623400371569}.

\bibitem{Wachsmuth2013:2}
{\sc G.~Wachsmuth}, {\em Strong stationarity for optimal control of the
  obstacle problem with control constraints}, SIAM Journal on Optimization, 24
  (2014), pp.~1914--1932, \url{https://doi.org/10.1137/130925827}.

\bibitem{Ye2005}
{\sc J.~J. Ye}, {\em Necessary and sufficient optimality conditions for
  mathematical programs with equilibrium constraints}, Journal of Mathematical
  Analysis and Applications, 307 (2005), pp.~350 -- 369,
  \url{https://doi.org/10.1016/j.jmaa.2004.10.032}.

\end{thebibliography}

\end{document}